\documentclass[a4paper,12pt]{amsart}
\usepackage[utf8]{inputenc}
\usepackage{amscd}
\usepackage{amsmath}
\usepackage{amsthm}

\usepackage{graphicx}

\usepackage{txfonts}
\usepackage{dsfont}
\usepackage{mathabx}

\usepackage{tikz-cd}
\usepackage[titletoc]{appendix}

\usepackage{hyperref}

\usepackage{geometry}
\usepackage{tikz}

\theoremstyle{definition}
\newtheorem{Defn}{Definition}[section]

\theoremstyle{plain}
\newtheorem{Lemma}[Defn]{Lemma}
\newtheorem{prop}[Defn]{Proposition}
\newtheorem{theorem}[Defn]{Theorem}
\newtheorem{Corollary}[Defn]{Corollary}

\theoremstyle{remark}
\newtheorem{remark}[Defn]{Remark}
\newtheorem{Example}[Defn]{Example}

\title{A pair of homotopy-theoretic version of TQFT's induced by a Brown functor}
\author{Minkyu Kim}
\date{}

\begin{document}

\maketitle
\begin{abstract}
The purpose of this paper is to study some obstruction classes induced by a construction of a homotopy-theoretic version of projective TQFT (projective HTQFT for short).
A projective HTQFT is given by a symmetric monoidal projective functor whose domain is the cospan category of pointed finite CW-spaces instead of a cobordism category.
We construct a pair of projective HTQFT's starting from a $\mathsf{Hopf}^\mathsf{bc}_k$-valued Brown functor where $\mathsf{Hopf}^\mathsf{bc}_k$ is the category of bicommutative Hopf algebras over a field $k$ : the cospanical path-integral and the spanical path-integral of the Brown functor.
They induce obstruction classes by an analogue of the second cohomology class associated with projective representations.
In this paper, we derive some formulae of those obstruction classes.
We apply the formulae to prove that the dimension reduction of the cospanical and spanical path-integrals are lifted to HTQFT's.
In another application, we reproduce the Dijkgraaf-Witten TQFT and the Turaev-Viro TQFT from an ordinary $\mathsf{Hopf}^\mathsf{bc}_k$-valued homology theory.
\end{abstract}
\tableofcontents


\section{Introduction}

The purpose of this paper is to study some obstruction classes induced by a construction of a homotopy-theoretic version of projective TQFT (a projective HTQFT for short).
A projective HTQFT's is given by a symmetric monoidal projective functor whose domain is the cospan category $\mathsf{Cosp}^\simeq_{\leq \infty} ( \mathsf{CW}^\mathsf{fin}_\ast )$ of pointed finite CW-spaces \cite{kim2020extension} instead of a cobordism category.
We construct a pair of projective HTQFT's starting from a $\mathsf{Hopf}^\mathsf{bc,vol}_k$-valued Brown functor $E$, which we call by {\it the cospanical path-integral} $\hat{\mathrm{PI}} (E)$ and {\it the spanical path-integral} $\check{\mathrm{PI}} (E)$.
Here, a $\mathsf{Hopf}^\mathsf{bc,vol}_k$-valued Brown functor is an assignment $E(K)$ of a finite-dimensional bisemisimple bicommutative Hopf algebra to a pointed finite CW-space $K$ satisfying Mayer-Vietoris axiom in an appropriate sense.
\begin{equation}\notag
\begin{tikzcd}
\{ \mathsf{Hopf}^\mathsf{bc,vol}_k\mathrm{-valued~}\mathrm{Brown~functors} \} \ar[r, shift left, "\hat{\mathrm{PI}}"] \ar[r, shift right, "\check{\mathrm{PI}}"']  & \{ \mathsf{C}_k\mathrm{-valued~}\mathrm{proj.~HTQFT's} \} .
\end{tikzcd}
\end{equation}
The codomain of those projective HTQFT's is a symmetric monoidal category $\mathsf{C}_k$ which is proved to be a pivotal category.
The category $\mathsf{C}_k$ contains $\mathsf{Hopf}^\mathsf{bc,vol}_k$, i.e. the category of finite-dimensional bisemisimple bicommutative Hopf algebras and Hopf homomorphisms, as its subcategory.

The (finite) path-integral \cite{DW} \cite{Wakui} \cite{FQ} \cite{SV} \cite{heuts2014ambidexterity} \cite{freed2009topological} \cite{monnier2015higher} and state-sum \cite{TV} \cite{BW} \cite{kirillov2010turaev} are formulated in various ways although they stem from the same idea.
Our construction is deeply motivated by the path-integral and state-sum, but a new technique is necessary since we do not use any triangulations or cellular structures on spaces ; moreover, we deal with Hopf algebras which might not be group Hopf algebras or function Hopf algebras over an arbitrary field.
We give a new approach by using a notion of {\it integrals along Hopf homomorphisms} \cite{kim2021integrals}.

The cospanical and spanical path-integrals are projective in the sense that it preserves compositions up to a scalar in $k^\ast = k \backslash \{ 0 \}$ whereas the induced mapping class group representation for each space is not projective (see Remark \ref{202102161038}).
By an analogue of the second cohomology class associated with projective representations, we obtain a pair of obstruction classes $\hat{\mathds{O}} (E)$ and $\check{\mathds{O}} (E)$ of the cospan category of spaces with coefficients in the multiplicative group $k^\ast$.
\begin{align}\notag
\hat{\mathds{O}} (E),~\check{\mathds{O}} (E) \in H^2 ( \mathsf{Cosp}^\simeq_{\leq \infty} ( \mathsf{CW}^\mathsf{fin}_\ast ) ; k^\ast ) 
\end{align}
In general, it is not obvious whether the obstruction classes vanish or not.
In this paper, we derive some formulae of the obstruction classes.
The first formula below follows from Corollary \ref{202006041143}.

\vspace{3mm}
{\bf Theorem.}
Let $E$ be a $\mathsf{Hopf}^\mathsf{bc,vol}_k$-valued Brown functor.
Then we have
\begin{align}\notag
\hat{\mathds{O}} ( \Sigma^\ast E) = \check{\mathds{O}} (E) 
\end{align}

where $\Sigma$ is the suspension functor.
\vspace{3mm}

A Brown functor of our main interest is a properly extensible $\mathsf{Hopf}^\mathsf{bc,vol}_k$-valued Brown functor which is introduced in this paper (see Definition \ref{202102171348}).
A $\mathsf{Hopf}^\mathsf{bc}_k$-valued homology theory \cite{kim2020homology} consists of a sequence of $\mathsf{Hopf}^\mathsf{bc}_k$-valued Brown functors.
It induces a (possibly, empty) family of $\mathsf{Hopf}^\mathsf{bc,vol}_k$-valued properly extensible Brown functors.
Some typical examples of a $\mathsf{Hopf}^\mathsf{bc}_k$-valued homology theory are obtained from generalized (co)homology theories such as ordinary (co)homology theories, MO-theory, stable (co)homotopy theory etc., via the group Hopf algebra functor or the function Hopf algebra functor ; or via the exponential functor \cite{touze} if the generalized (co)homology theory is valued at the category of vector spaces.
On the one hand, one could obtain a properly extensible $\mathsf{Hopf}^\mathsf{bc,vol}_k$-valued Brown functor from the loop space of a homotopy commutative H-group subject to a finite condition on its homotopy groups.
For the details, see subsection \ref{202102141031}.
For a properly extensible Brown functor, we have the following formula which follows from Theorem \ref{202001271630}.

\vspace{3mm}
{\bf Theorem.}
If a $\mathsf{Hopf}^\mathsf{bc,vol}_k$-valued Brown functor $E$ is properly extensible, then we have
\begin{align}\notag
\hat{\mathds{O}} (E) = \check{\mathds{O}} (E)^{-1} .
\end{align}
\vspace{0.1mm}

Note that the main theorems above are refined by considering {\it $d$-dimensional Brown functors} for $d \in \mathbb{N} \amalg \{ \infty \}$.

In application of the main theorems, we show that the tensor product of the cospanical and spanical path-integrals of $E$, $\hat{\mathrm{PI}} (E) \otimes \check{\mathrm{PI}} (E)$, is lifted to a HTQFT $Z$ whose associated homotopy invariant $Z(L)$ is the dimension of the Hopf algebra $E(L)$ (see Corollary \ref{202101301508}).

In the literature of topological field theory, the cartesian product of manifolds with a circle $\mathbb{T}$ induces the dimension reduction of topological field theories (for example see \cite{freed2009remarks}).
We deal with its pointed version by considering the smash product with $\mathbb{T}^+ = \mathbb{T} \amalg \{ \mathrm{pt} \}$.
We give an interesting phenomenon related with the dimension reduction.
By applying the main theorems, we show that the dimension reduction of the cospanical and spanical path-integrals are lifted to a HTQFT.

In the last application, we show that the obstruction classes associated with a bounded-below (or -above) homology theory vanish.
Furthermore, we construct a canonical solution of the associated coboundary equation.
By {\it the first ordinary homology theory} with coefficients in an appropriate bicommutative Hopf algebra, these results reproduce the untwisted Dijkgraaf-Witten-Freed-Quinn TQFT \cite{DW} \cite{Wakui} \cite{FQ} \cite{SV} \cite{heuts2014ambidexterity} \cite{freed2009topological} \cite{monnier2015higher} of abelian groups and Turaev-Viro-Barrett-Westbury TQFT \cite{TV} \cite{BW} \cite{kirillov2010turaev} of bicommutative Hopf algebras.

The results in this paper hold for cohomology theories in a dual way.

The vector spaces assigned to surfaces by TVBW theory are naturally isomorphic to the 0-eigenspaces (physically, the ground-state spaces) in the Kitaev lattice Hamiltonian model (a.k.a. {\it toric code}) \cite{BK}\cite{Kit}\cite{BMCA}.
In \cite{kim2019bicommutative}, we give a generalization of the Kitaev lattice Hamiltonian model based on chain complex theory of bicommutative Hopf algebras.
In that paper, we formulate {\it topological local stabilizer models} and prove that the assignment of the associated 0-eigenspaces extends to a $\mathsf{Hopf}^\mathsf{bc,vol}_k$-valued Brown functor.
As a consequence of this paper, the eigenspaces extend to a projective HTQFT which is lifted to a HTQFT under some conditions.
In other words, the relationship between TVBW theory and Kitaev lattice Hamiltonian model is generalized to an arbitrary ground field $k$ and pointed finite CW-complexes.

This paper is organized as follows.
In subsection \ref{202002070916}, we give a review of our another previous study about integrals along Hopf homomorphisms \cite{kim2021integrals} .
In subsection \ref{202102171833}, we introduce a notion of integrals along (co)spans in $\mathsf{Hopf}^\mathsf{bc,vol}_k$.
In section \ref{202002221748}, we introduce the symmetric monoidal category $\mathsf{C}_k$ and prove that $\mathsf{C}_k$ is a pivotal category.
In subsection \ref{202102141031}, we give an explanation of $\mathcal{A}$-valued Brown functors with some examples.
In subsection \ref{202102021122}, we give an overview of cospanical and spanical extensions \cite{kim2020extension}.
In subsection \ref{202102171836}, we define the notion of homotopy-theoretic version of TQFT's.
In subsection \ref{202102131116}, we construct the cospanical and spanical path-integrals of Brown functors.
In section \ref{202002221749}, we study the obstruction classes associated with the integral projective functors.
In subsection \ref{202002230946}, we give basic properties of the obstruction classes associated with the path-integrals of Brown functors.
In subsection \ref{202102150945}, we derive some formulae for properly extensible Brown functors which we call the inversion formulae.
In subsection \ref{202002211404}, we prove that the obstruction classes associated with the dimension reduction of the path-integrals vanish.
In subsection \ref{202001210015}, we prove that the obstruction classes induced by a bounded-below homology theory vanish.
Moreover we reproduce the abelian DW TQFT and TV TQFT from the results.
In appendix \ref{202002070917}, we give an overview of symmetric monoidal projective functors and the obstruction class.

\section*{Acknowledgements}
The author was supported by FMSP, a JSPS Program for Leading Graduate Schools in the University of Tokyo, and JPSJ Grant-in-Aid for Scientific Research on Innovative Areas Grant Number JP17H06461.

\section{Integrals associated with Hopf homomorphisms}

In this section, we introduce an integral along a cospan of bicommutative Hopf algebras with a finite volume based on the results in \cite{kim2021integrals}.
It is necessary to formulate the path-integral of Brown functors in subsection \ref{202102131116}.

\subsection{Integrals along Hopf homomorphisms}
\label{202002070916}

In this subsection, we give an overview of the results in \cite{kim2021integrals} where the results are given based on a symmetric monoidal category $\mathcal{C}$ satisfying some assumptions.
Such assumptions automatically hold for $\mathcal{C} = ( \mathsf{Vec}_k , \otimes_k )$, the tensor category of vector spaces over a field $k$.
Here, we give those results in this specific case.
For a Hopf homomorphism $\xi :A \to B$, {\it a normalized generator integral along $\xi$} is a linear homomorphism $\mu : B \to A$ satisfying some axioms (see Definition 3.4, 3.11 in \cite{kim2021integrals}).

\begin{Example}
Let $A$ be a Hopf algebra with a normalized integral $\sigma_A \in A$, i.e. $\sigma_A \cdot a = a \cdot \sigma_ A = \epsilon_A ( a ) \cdot \sigma_ A$ and $\epsilon_A ( \sigma_A )= 1$.
If we consider $\sigma_A$ as a linear homomorphism from $k$ to $A$, then $\sigma_A$ is a normalized generator integral along the counit $\epsilon_A$.
Similarly, a normalized cointegral $\sigma^A : A \to k$ is a normalized generator integral along the unit $\eta_A : k \to A$.
\end{Example}

For a group $G$, we denote by $kG$ {\it the group Hopf} of $G$.
Its underlying vector space is generated by the elements of $G$, the multiplication is given by the group operation and the comultiplication is induced by the diagonal map on $G$.
Dually, we denote by $k^G$ {\it the function Hopf algebra} of $G$, i.e. the dual Hopf algebra of $kG$.

\begin{Example}
Let $\varrho : G \to H$ be a group homomorphism and $\xi : \varrho_\ast : kG \to k H$ be the induced Hopf homomorphism between group Hopf algebras.
Then the normalized generator integral $\mu_\xi : kH \to kG$ is given by $\mu_\xi (h) =  | \mathrm{Ker} ( \varrho ) |^{-1} \cdot \sum_{\varrho (g) = h} g$.
\end{Example}

\begin{Example}
\label{202102112335}
Let $\xi = \varrho^\ast : k^H \to k^G$ be the induced Hopf homomorphism between function Hopf algebras.
Then the normalized generator integral $\mu_\xi : k^G \to k^H$ is given by $(\mu_\xi (f))(h) = | \mathrm{Ker} ( \varrho ) |^{-1} \cdot \sum_{\varrho (g) = h} f (g)$.
In particular, if $f$ is the unit of the Hopf algebra $k^G$, then $(\mu_\xi (f))(h)$ is $1$ if $\varrho^{-1} (h) \neq \emptyset$ and $0$ otherwise.
\end{Example}

\begin{theorem}
\label{202002211508}
Let $A,B$ be bicommutative Hopf algebras and $\xi : A \to B$ be a Hopf homomorphism.
There exists a normalized generator integral $\mu_\xi$ along $\xi$ if and only if the following conditions hold :
\begin{enumerate}
\item
the kernel Hopf algebra $\mathrm{Ker_H} (\xi)$ has a normalized integral.
\item
the cokernel Hopf algebra $\mathrm{Cok_H} (\xi)$ has a normalized cointegral.
\end{enumerate}
Note that if a normalized integral along $\xi$ exists, then it is unique.
\end{theorem}

\begin{Defn}
We introduce an invariant of bicommutative Hopf algebras $A$, called an inverse volume $\mathrm{vol}^{-1} (A)$.
It is defined as a composition $\sigma^A \circ \sigma_A \in k$ where $\sigma_A$ is a normalized integral and $\sigma^A$ is a normalized cointegral.
\end{Defn}

\begin{Defn}
\label{202001271558}
A bicommutative Hopf algebra $A$ {\it has a finite volume} if 
\begin{enumerate}
\item
It has a normalized integral $\sigma_A : k \to A$.
\item
It has a normalized cointegral $\sigma^A : A \to k$.
\item
Its inverse volume $\mathrm{vol}^{-1}( A) = ( \sigma^A \circ \sigma_A ) \in k$ is invertible.
\end{enumerate}
Denote by $\mathsf{Hopf}^\mathsf{bc,vol}_k$ the category of bicommutative Hopf algebras with a finite volume and Hopf homomorphisms.
\end{Defn}

\begin{remark}
See Corollary \ref{202006091350} for an equivalent description.
\end{remark}

As a corollary of Theorem \ref{202002211508}, we obtain the following statement.

\begin{Corollary}
\label{202001251639}
Let $A,B$ be bicommutative Hopf algebras with a finite volume.
For any Hopf homomorphism $\xi : A \to B$, there exists a unique normalized generator integral $\mu_\xi$ along $\xi$.
\end{Corollary}

We make use of the following proposition in this paper.

\begin{prop}
\label{202006030739}
Let $\xi : A \to B$ be a Hopf homomorphism between bicommutative Hopf algebras with a finite volume.
\begin{enumerate}
\item
If $\xi$ is an epimorphism in the category $\mathsf{Hopf}^\mathsf{bc}_k$, then we have $\xi \circ \mu_\xi = id_B$.
In other words, $\mu_\xi$ is a section of $\xi$ in the category $\mathsf{Vec}_k$.
\item
If $\xi$ is an monomorphism in the category $\mathsf{Hopf}^\mathsf{bc}_k$, then we have $\mu_\xi \circ \xi = id_A$.
In other words, $\mu_\xi$ is a retract of $\xi$ in the category $\mathsf{Vec}_k$.
\end{enumerate}
\end{prop}
\begin{proof}
It is immediate from Lemma 9.3 \cite{kim2021integrals}.
\end{proof}

The inverse volume induces a volume on the abelian category $\mathsf{Hopf}^\mathsf{bc,bs}_k$ consisting of bicommutative Hopf algebras with a normalized integral and a normalized cointegral.
Here, the volume on the abelian category is a generalization of the dimension of vector spaces and the order of abelian groups, which is also introduced in \cite{kim2021integrals}.

\begin{theorem}
\label{202002211501}
We regard the field $k$ as the multiplicative monoid.
Then the inverse volume $\mathrm{vol}^{-1}$ gives an $k$-valued volume on the abelian category $\mathsf{Hopf}^\mathsf{bc,bs}_k$, i.e. if $A \to B \to C$ is a short exact sequence in $\mathsf{Hopf}^\mathsf{bc,bs}_k$, then we have $\mathrm{vol}^{-1} (B) = \mathrm{vol}^{-1} (A) \cdot \mathrm{vol}^{-1} (C)$.
\end{theorem}

By Theorem \ref{202002211501}, $\mathsf{Hopf}^\mathsf{bc,vol}_k \subset \mathsf{Hopf}^\mathsf{bc,bs}_k$ is closed under short exact sequences.
In particular, $\mathsf{Hopf}^\mathsf{bc,vol}_k$ is also an abelian category.
Then the following corollary is immediate from Theorem \ref{202002211501}.

\begin{Corollary}
\label{202002211513}
The inverse volume $\mathrm{vol}^{-1}$ gives an $k^\ast$-valued volume on the abelian category $\mathsf{Hopf}^\mathsf{bc,vol}_k$.
Here, we regard $k^\ast = k \backslash \{ 0 \}$ as the multiplicative group.
\end{Corollary}

\begin{prop}
\label{202102071242}
Consider the exact square diagram (\ref{201912190830}) for $\mathcal{A} = \mathsf{Hopf}^\mathsf{bc,vol}_k$.
Then we have
\begin{align}
\mu_{g^\prime} \circ g = f^\prime \circ \mu_f .
\end{align}
\end{prop}
\begin{proof}
It follows from Theorem 10.1 in \cite{kim2021integrals}.
Note that an epimorphism $\pi$ in the category $\mathsf{Hopf}^\mathsf{bc,vol}_k$ has a section in $\mathsf{Vec}_k$.
In fact, the normalized integral $\mu_\pi$ along $\pi$ is a section of $\pi$ in $\mathsf{Vec}_k$ by Proposition \ref{202006030739}.
Similarly, any monomorphism in the category $\mathsf{Hopf}^\mathsf{bc,vol}_k$ has a retract in in $\mathsf{Vec}_k$. 
\end{proof}

The inverse volume of bicommutative Hopf algebras is generalized to Hopf homomorphisms.
For a Hopf homomorphism $\xi : A \to B$, we define $\langle \xi \rangle = \sigma^B \circ \xi \circ \sigma_A \in k$.
By using this notion, a composition rule of normalized integrals is represented as follows.

\begin{Defn}
\label{202101301421}
Let $\xi : A \to B$ and $\xi^\prime : B \to C$ be Hopf homomorphisms.
Consider a cokernel Hopf algebra $\mathrm{Cok_H} ( \xi )$ with the canonical homomorphism $\mathrm{cok_H} ( \xi ) : B \to \mathrm{Cok_H} ( \xi )$ and a kernel Hopf algebra $\mathrm{Ker_H} ( \xi^\prime )$ with the canonical homomorphism $\mathrm{ker_H} ( \xi^\prime ) : \mathrm{Ker_H} ( \xi^\prime ) \to B$.
We define a Hopf homomorphism $\partial ( \xi^\prime , \xi ) : \mathrm{Ker_H} ( \xi^\prime) \to \mathrm{Cok_H} ( \xi )$ by $\partial ( \xi^\prime , \xi ) \stackrel{\mathrm{def.}}{=} \mathrm{cok_H} ( \xi ) \circ \mathrm{ker_H} ( \xi^\prime )$.
\end{Defn}

\begin{prop}
\label{202005301216}
Let $\xi : A \to B, \xi^\prime : B \to C$ be morphisms in the category $\mathsf{Hopf}^\mathsf{bc,vol}_k$.
Then there exists a unique $\lambda \in k^\ast$ such that
\begin{align}
\mu_\xi \circ \mu_{\xi^\prime} = \lambda \cdot \mu_{\xi^\prime\circ\xi} .
\end{align}
Moreover, we have $\lambda = \langle \partial ( \xi^\prime , \xi ) \rangle$.
\end{prop}
\begin{proof}
It follows from Theorem 12.1 in \cite{kim2021integrals}.
\end{proof}

\subsection{Integrals along (co)spans}
\label{202102171833}

In this subsection, we study an invariant of cospans in the category $\mathsf{Hopf}^\mathsf{bc,vol}_k$, called an integral along cospans.
In a parallel way, an integral along spans is defined and we give a comparison proposition.

\begin{Defn}
\label{202001110741}
Let $A_0, A_1$ be bicommutative Hopf algebras with a finite volume.
Consider a cospan $\Lambda = \left( A_0 \stackrel{\xi_0}{\to} B \stackrel{\xi_1}{\leftarrow} A_1 \right)$ in the category $\mathsf{Hopf}^\mathsf{bc,vol}_k$.
We define {\it an integral along $\Lambda$} by a linear homomorphism from $A_0$ to $A_1$,
\begin{align}
\int_\Lambda \stackrel{\mathrm{def.}}{=} \mu_{\xi_1} \circ \xi_0 .
\end{align}
Here, $\mu_{\xi_1}$ is the normalized integral along $\xi_1$.
A linear homomorphism $\varrho : A_0 \to A_1$ is {\it realized as a nontrivial integral along a cospan} if there exists a cospan $\Lambda$ in the category $\mathsf{Hopf}^\mathsf{bc,vol}_k$ and $\lambda \in k^\ast$ such that $\varrho = \lambda \cdot \int_\Lambda$.
\end{Defn}

\begin{Example}
Let $\Lambda = \left( k \stackrel{\eta_B}{\to} B \stackrel{\eta_B}{\leftarrow} k \right)$ be the cospan induced by the unit of $B$.
The normalized integral along $\eta_B$ is the normalized cointegral of $B$ so that $\int_\Lambda$ is the identity on the trivial Hopf algebra $k$.
\end{Example}

\begin{prop}
\label{202001241622}
If $\varrho : A_0 \to A_1$ and $\varrho^\prime : A_1 \to A_2$ are realized as a nontrivial integral along a cospan, then the composition $\varrho^\prime \circ \varrho$ is realized as a nontrivial integral along a cospan.
\end{prop}
\begin{proof}
Suppose that $\varrho = \lambda \cdot \int_\Lambda$ and $\varrho^\prime = \lambda^\prime \cdot \int_{\Lambda^\prime}$ for some $\lambda,\lambda^\prime \in k^\ast$ and some cospans $\Lambda,\Lambda^\prime$ in the category $\mathsf{Hopf}^\mathsf{bc,vol}_k$.
Let $\Lambda = \left( A_0 \stackrel{\xi_0}{\to} B \stackrel{\xi_1}{\leftarrow} A_1 \right)$ and $\Lambda^\prime = \left( A_1 \stackrel{\xi^\prime_1}{\to} B^\prime \stackrel{\xi^\prime_2}{\leftarrow} A_2 \right)$.
Recall that the composition $\Lambda^\prime \circ \Lambda$ is defined by $\left( A_0 \stackrel{\varphi\circ\xi_0}{\to} B^{\prime\prime} \stackrel{\varphi^\prime\circ\xi^\prime_2}{\leftarrow} A_2 \right)$ where $B^{\prime\prime}$ is given by the pushout diagram (\ref{201912251725}).
We obtain $\int_{\Lambda^\prime} \circ \int_\Lambda = \mu_{\xi^\prime_2} \circ \mu_{\varphi^\prime} \circ \varphi \circ \xi_0$.
Since we have $\mu_{\xi^\prime_2} \circ \mu_{\varphi^\prime} = \lambda^{\prime\prime} \cdot \mu_{\varphi^\prime \circ \xi^\prime_2}$ for $\lambda^{\prime\prime} = \langle cok ( \xi^\prime_2 ) \circ ker (\varphi^\prime ) \rangle \in k^\ast$ by Proposition \ref{202005301216}, we obtain $\int_{\Lambda^\prime} \circ \int_\Lambda = \lambda^{\prime\prime} \cdot \int_{\Lambda^\prime \circ \Lambda}$, hence $\varrho^\prime \circ \varrho = \lambda^{\prime\prime\prime} \cdot \int_{\Lambda^\prime \circ \Lambda}$ where $\lambda^{\prime\prime\prime} = \lambda^{\prime\prime}\cdot \lambda^\prime \cdot \lambda$.
By definition, the composition $\varrho^\prime \circ \varrho$ is realized as a nontrivial integral along a cospan.
\begin{equation}
\label{201912251725}
\begin{tikzcd}
& B^{\prime\prime}  & \\
B \ar[ur, "\varphi"] & & B^\prime \ar[ul, "\varphi^\prime"'] \\
& A_1 \ar[ul, "\xi_1"] \ar[ur, "\xi^\prime_1"'] &
\end{tikzcd}
\end{equation}
\end{proof}

\begin{theorem}[Invariance of integrals]
\label{201912302324}
An equivalence $\Lambda \approx \Lambda^\prime$ (see subsection \ref{202102021122}) implies $\int_{\Lambda} = \int_{\Lambda^\prime}$.
\end{theorem}
\begin{proof}
By the definition of $\approx$, it suffices to prove that a preorder $\Lambda \preceq \Lambda^\prime$ implies $\int_{\Lambda} = \int_{\Lambda^\prime}$ (see Definition \ref{201912112309}).
The inverse volume $\langle \mathrm{cok_H} (f_1) \circ \mathrm{ker_H}(g) \rangle =\sigma^{\mathrm{Cok_H}(f_1)} \circ \mathrm{cok_H} (f_1) \circ \mathrm{ker_H} (g) \circ \sigma_{\mathrm{Ker_H}(g)}$ is $1 \in k$ since $\mathrm{Ker_H}(g) \cong k$ and $\sigma^{\mathrm{Cok_H}(f_1)}$ is a normalized cointegral.
By Proposition \ref{202005301216}, we have $\mu_{f_1} \circ \mu_g = \mu_{g \circ f_1}$.
Hence, we obtain $\int_{\Lambda^\prime} = \mu_{g \circ f_1} \circ g \circ f_0 = \mu_{f_1} \circ \mu_g \circ g \circ f_0$.
By the second part in Proposition \ref{202006030739}, we have $\mu_g \circ g = Id_{B}$ so that we obtain $\int_{\Lambda^\prime} = \mu_{f_1} \circ f_0 = \int_{\Lambda}$.
It completes the proof.
\end{proof}

We introduce {\it an integral along a span} as an analogue of Definition \ref{202001110741}.

\begin{Defn}
Let $\mathrm{V} = \left( A_0 \stackrel{\xi_0}{\leftarrow} B \stackrel{\xi_1}{\to} A_1 \right)$ be a span in the category $\mathsf{Hopf}^\mathsf{bc,vol}_k$.
{\it An integral along $\mathrm{V}$} is defined by $\int^\mathrm{V} = \xi_1 \circ \mu_{\xi_0}$.
\end{Defn}

\begin{prop}
\label{202102071235}
For a span $\mathrm{V}$, we have
\begin{align}
\label{202102071251}
\int^\mathrm{V} = \int_{\mathrm{V}^\mathrm{T}} .
\end{align}
Here, $\mathrm{T}$ is the transposition between cospans and spans (see (\ref{202102071200}) in subsection \ref{202102021122}).
Especially, a linear homomorphism is realized by a nontrivial integral along a cospan if and only if it is realized by a nontrivial integral along a span.
\end{prop}
\begin{proof}
It is a result of a simple application of Proposition \ref{202102071242}.
\end{proof}


\section{The category $\mathsf{C}_k$}
\label{202002221748}

In this section, we introduce a symmetric monoidal category $\mathsf{C}_k$ consisting of bicommutative Hopf algebras with a finite volume.
It is a category in which integrals along cospans in $\mathsf{Hopf}^\mathsf{bc,vol}_k$ could lie without losing the Hopf algebra structures of the source and target.
Moreover, we prove that the category is a pivotal category.

\begin{Defn}
For a field $k$, we introduce a category $\mathsf{C}_k$ of bicommutative Hopf algebras over $k$ with a finite volume.
Consider a category $\mathcal{N}$ of bicommutative Hopf algebras with a finite volume defined as follows.
For two objects $A,B$ of $\mathcal{N}$, the morphism set $Mor_{\mathcal{N}} ( A, B )$ consists of linear homomorphisms.
Then the category $\mathsf{C}_k$ is the smallest subcategory of $\mathcal{N}$ which contains all the linear homomorphism realized as a nontrivial integral along a cospan.
Equivalently, it is the smallest subcategory which contains the following three classes of morphisms :
\begin{itemize}
\item
a Hopf homomorphism $\xi : A \to B$ for objects $A,B$ of $\mathcal{N}$,
\item
a morphism $\mu : A \to B$ in $\mathcal{N}$ which is a normalized integral along some Hopf homomorphism $\xi : B \to A$,
\item
an automorphism on the unit object $k$ in $\mathcal{N}$.
\end{itemize}
\end{Defn}

\begin{prop}
Any morphism in $\mathsf{C}_k$ is realized as a nontrivial integral along cospans. 
\end{prop}
\begin{proof}
By definitions, any morphism in $\mathsf{C}_k$ is given by a composition of linear homomorphisms realized as nontrivial integrals along cospans.
By Proposition \ref{202001241622}, such a composition is reduced to only one linear homomorphism realized as a nontrivial integral along a cospan.
\end{proof}

\begin{Defn}
\label{202006081524}
Let $A$ be a bicommutative Hopf algebra with a finite volume.
Let $\Lambda = \left( k \stackrel{\eta}{\to} A \stackrel{\nabla}{\leftarrow} A \otimes A \right)$ be a cospan in $\mathsf{Hopf}^\mathsf{bc,vol}_k$.
We define morphisms $\mathrm{i}_A : k \to A \otimes A$ and $\mathrm{e}_A : A \otimes A \to k$ in $\mathsf{C}_k$ by
\begin{align}
\mathrm{i}_A &\stackrel{\mathrm{def.}}{=} \int_\Lambda , \\
\mathrm{e}_A &\stackrel{\mathrm{def.}}{=} a^{-1} \cdot \int_{\Lambda^\dagger} .
\end{align}
Here, $a = \mathrm{vol}^{-1} (A) \in k^\ast$ is the inverse volume of $A$.
\end{Defn}

\begin{prop}
\label{202005301300}
The morphisms $\mathrm{i}_A, \mathrm{e}_A$ give a symmetric self-duality of $A$ in $\mathsf{C}_k$.
\end{prop}
\begin{proof}
Since $A$ is bicommutative, we have $\tau \circ \mathrm{i}_A = \mathrm{i}_A$, $\mathrm{e}_A \circ \tau = \mathrm{e}_A$ where $\tau : A \otimes A \to A \otimes A ; x \otimes y \mapsto y \otimes x$.
All that remain is to prove that $\mathrm{i}_A , \mathrm{e}_A$ form a duality.
Let $e^\prime_A = \int_{\Lambda^\dagger}$.
Then a zigzag diagram is computed as Figure \ref{202005300947}.
Note that in the third equation we use the definition of the normalized integral $\mu_\nabla$.
The third equality holds since $\mu_\nabla$ is an integral along the multiplication $\nabla$.
Note that the normalized cointegral $\sigma^A$ is a normalized integral along the unit $\eta : k \to A$.
The last morphism $(id_A \otimes \sigma_A) \circ \mu_\nabla : A \to A$ is proportional to a normalized integral along the identity on $A$.
The proportional factor coincides with $\mathrm{vol}^{-1} (A)$ due to Proposition \ref{202005301216}.
It completes the proof since $\mathrm{i}_A,\mathrm{e}_A$ are symmetric.
\begin{figure}[h]
  \includegraphics[width=\linewidth]{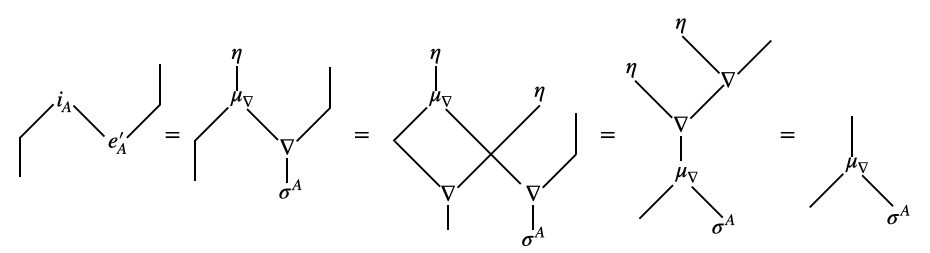}
  \caption{A part of the proof of Proposition \ref{202005301300}}
  \label{202005300947}
\end{figure}
\end{proof}

\begin{Corollary}
\label{202101232004}
The symmetric monoidal category $\mathsf{C}_k$ is a pivotal category.
Moreover, the dimension of an object $A$ in the pivotal category coincides with $\mathrm{vol}^{-1} (A)^{-1}$.
\end{Corollary}
\begin{proof}
By Proposition \ref{202005301300}, all the objects of $\mathsf{C}_k$ are rigid, i.e. there exist dualities.
Moreover, the monoidal natural isomorphism $(A^\ast)^\ast \cong A$ (the pivotal structure) is given by the identity since all the objects have symmetric self-dualities.
Here, we denote by $A^\ast$ the dual object of $A$ in $\mathsf{C}_k$.

All that remain is to compute the dimension $\mathrm{e}_A \circ \mathrm{i}_A$.
We use the fact that for a morphsim $\xi : A \to B$ in $\mathsf{Hopf}^\mathsf{bc,vol}_k$, if $\xi$ is an epimorphism in $\mathsf{Hopf}^\mathsf{bc,vol}_k$ then we have $\xi \circ \mu_\xi = id_B$.
It follows from Proposition \ref{202006030739}.
Note that $\nabla : A \otimes A \to A$ is an epimorphism in $\mathsf{Hopf}^\mathsf{bc,vol}_k$.
Hence we obtain,
\begin{align}
\mathrm{vol}^{-1} (A) \cdot \left( \mathrm{e}_A \circ \mathrm{i}_A \right) &= \sigma^A \circ \nabla \circ \mu_\nabla \circ \eta , \\
&= \sigma^A \circ \eta , \\
&= 1 .
\end{align}
\end{proof}

\begin{Corollary}
\label{202006040804}
A bicommutative Hopf algebra with a finite volume is finite-dimensional.
Moreover, we have $\mathrm{vol}^{-1}(A)^{-1}= ( \dim_k A ) \cdot 1 \in k$.
\end{Corollary}
\begin{proof}
Note that the forgetful functor from $\mathsf{C}_k$ to $\mathsf{Vec}_k$ is a symmetric monoidal functor.
In particular, a rigid object in $\mathsf{C}_k$ with a distinguished duality is preserved via the forgetful functor.
The first claim is proved by Proposition \ref{202005301300} since a vector space with a duality is finite-dimensional.
Moreover the dimension of a vector space computed by a duality coincides with $( \dim_k A ) \cdot 1 \in k$ so that the second claim follows from Corollary \ref{202101232004}.
\end{proof}

\begin{Corollary}
\label{202006091350}
Let $A$ be a bicommutative Hopf algebra.
Then the following conditions are equivalent.
\begin{enumerate}
\item
$A$ has a finite volume
\item
$A$ is finite-dimensional and bisemisimple i.e. semisimple and cosemisimple.
\item
$A$ is finite-dimensional and $(\dim_k A ) \cdot 1 \in k$ is nonzero.
\end{enumerate}
\end{Corollary}
\begin{proof}
We first prove that (1) and (2) are equivalent.
A finite-dimensional Hopf algebra is semisimple if and only if it has a normalized integral \cite{LarSwe}.
In the same manner, the cosemisimplicity of a finite-dimensional Hopf algebra is equivalent with an existence of a normalized cointegral.
In Theorem 3.3 \cite{BKLT}, it is proved that the composition of a left (right) $Int A$-valued integral and a left (right) $Int A$-based integral of finite-dimensional Hopf algebra is invertible.
Hence, a bicommutative Hopf algebra $A$ has a finite volume if it is finite-dimensional, semisimple and cosemisimple.
Conversely, (2) follows from (1) due to Corollary \ref{202006040804} and the above discussion.
The equivalence of (2) and (3) follows from Theorem 2.5 (b) in \cite{larson1988finite} since a bicommutative Hopf algebra $A$ is involutory.
\end{proof}


\section{Some extensions of Brown functors}
\label{202002211539}

In this section, we give an overview of \cite{kim2020extension}.

\subsection{$\mathcal{A}$-valued Brown functor}
\label{202102141031}

In this subsection, we give an explanation of Brown functors.
Brown functor is a well-studied object which is a functor from a homotopy category of pointed CW-spaces to an opposite category of pointed sets satisfying some axioms \cite{hatcher2002algebraic}.
For our convenience, we give a refinement of the definition by the dimension of CW-spaces and replace the target category with an abelian category $\mathcal{A}$.

\begin{Defn}

A CW-complex is a topological space consisting of so-called cells in an appropriate way \cite{hatcher2002algebraic}.
A topological space is a CW-space if it {\it has} a CW-complex structure.
Note that we do not fix any distinguished CW-complex structure.
As an analogue of the dimension of manifolds, the dimension of CW-complex is defined as the top dimension of cells.
The dimension of CW-space is the dimension of some distinguished CW-complex structure ; it is easy to see that it is well-defined.
We denote by $\mathsf{Ho} ( \mathsf{CW}^\mathsf{fin}_{\ast , \leq d} )$ the homotopy category of pointed finite CW-spaces $K$ with $\dim K \leq d$.
\end{Defn}

\begin{Defn}
\label{201912190819}
\rm
Consider a diagram in $\mathsf{CW}^\mathsf{fin}_{\ast,\leq r}$ which commutes up to a homotopy :
\begin{equation}
\label{201912231254}
\begin{tikzcd}
K_0 \ar[r, "f_0"] & L \\
T \ar[r, "g_1"'] \ar[u, "g_0"] & K_1 \ar[u, "f_1"'] 
\end{tikzcd}
\end{equation}
The diagram (\ref{201912231254}) is {\it approximated by a triad of spaces} if there exists a triad of pointed finite CW-spaces $(L^\prime , K^\prime_0 , K^\prime_1)$ such that the following induced diagram (\ref{201912231255}) is homotopy equivalent with the diagram (\ref{201912231254}).
In other words, there exist pointed homotopy equivalences $K_0 \simeq K^\prime_0$, $K_1 \simeq K^\prime_1$, $T \simeq K^\prime \cap K^\prime_1$ and $L \simeq K^\prime_0 \cup K^\prime_1$ such that the diagrams (\ref{201912231254}) and (\ref{201912231255}) coincide up to homotopies.
Note that we do not restrict the dimensions of $L^\prime, K^\prime_0 , K^\prime_1$ whereas $L,K_0,K_1,T$ are $r$-dimensional at most.
\begin{equation}
\label{201912231255}
\begin{tikzcd}
K^\prime_0 \ar[r, hookrightarrow] & K^\prime_0 \cup K^\prime_1 \\
K^\prime_0 \cap K^\prime_1 \ar[r, hookrightarrow] \ar[u, hookrightarrow] & K^\prime_1 \ar[u, hookrightarrow] 
\end{tikzcd}
\end{equation}
\end{Defn}

\begin{Defn}
\label{202101301425}
Let $\mathcal{A}$ be an abelian category.
A {\it square diagram} is a commutative diagram $\Box$ in $\mathcal{A}$ as follows :
\begin{equation}
\label{201912190830}
\begin{tikzcd}
B \ar[r, "g"] & D \\
A \ar[r , "f^\prime"'] \ar[u, "f"] &  C  \ar[u, "g^\prime"']
\end{tikzcd}
\end{equation}
The morphism $f$ induces a morphism $k_\Box : Ker_\mathcal{A} (f^\prime ) \to Ker_\mathcal{A} ( g)$.
The morphism $g^\prime$ induces a morphism $c_\Box : Cok_\mathcal{A} ( f^\prime ) \to Cok_\mathcal{A} ( g)$.
The square diagram is {\it exact} if the morphism $k_\Box$ is an epimorphism and the morphism $c_\Box$ is a monomorphism.
\end{Defn}

\begin{Defn}
\label{202005211445}
{\it An $\mathcal{A}$-valued $r$-dimensional Brown functor} is a functor $E : \mathsf{Ho} ( \mathsf{CW}^\mathsf{fin}_{\ast,\leq r} ) \to \mathcal{A}$ which sends the wedge sum to the biproduct in $\mathcal{A}$ and satisfies the Mayer-Vietoris axiom :
if for an arbitrary diagram (\ref{201912231254}) in $\mathsf{CW}^\mathsf{fin}_{\ast,\leq r}$ approximated by a triad of spaces, the induced square diagram in $\mathcal{A}$ is exact.
\begin{equation}
\begin{tikzcd}
E(K_0) \ar[r, "E(f_0)"] & E(L) \\
E(T) \ar[r , "E(g_1)"'] \ar[u, "E(g_0)"] &  E(K_1)  \ar[u, "E(f_1)"']
\end{tikzcd}
\end{equation}
Equivalently, the induced chain complex is exact by Remark \ref{201912230956}.
An $\mathcal{A}$-valued $\infty$-dimensional Brown functor is {\it an $\mathcal{A}$-valued Brown functor} for short.
\end{Defn}

\begin{remark}
\label{201912230956}
The exactness of a square diagram is represented in a familiar way as follows.
See \cite{kim2020extension} for the proof.
Consider the square diagram (\ref{201912190830}).
We define {\it the induced chain complex} $C(\Box)$ by
\begin{align}
A \stackrel{u_\Box}{\to} B \oplus C \stackrel{v_\Box}{\to} D
\end{align} 
where $u_\Box \stackrel{\mathrm{def.}}{=} ( f \oplus (-f^\prime) ) \circ \Delta_A$ and $v_\Box \stackrel{\mathrm{def.}}{=} \nabla_D \circ (g \oplus g^\prime)$ where $\Delta_A$ denotes the diagonal map and $\nabla_D$ denotes the sum operation.
Then the following conditions are equivalent :
\begin{enumerate}
\item
The square diagram $\Box$ is exact.
\item
The induced chain complex $C(\Box)$ is exact.
\end{enumerate}
Even though these conditions are equivalent, we make use of the exactness of square diagrams to define Brown functor in Definition \ref{202005211445} since a usual Brown functor is defined based on such a square diagram by requiring that it sends a pushout to a weak pullback.
If $\mathcal{A}= \mathsf{Ab}^\mathsf{op}$, the opposite category of abelian groups, then the definitions are equivalent.
Furthermore, the first condition is more convenient to study the cospanical and spanical extensions in \cite{kim2020extension} and path-integrals in subsection \ref{202102131116}.
\end{remark}

\begin{remark}
Recall that the target category of a usual Brown functor is the opposite category of pointed sets $\mathsf{Sets}^\mathsf{op}_\ast$ \cite{hatcher2002algebraic}.
We consider an abelian category $\mathcal{A}$ instead of $\mathsf{Sets}^\mathsf{op}_\ast$ since our main interest is $\mathcal{A} = \mathsf{Hopf}^\mathsf{bc,vol}_k$, i.e. the category of bicommutative Hopf algebras with a finite volume.
We essentially use the abelian property to show the existence of integrals in \cite{kim2021integrals} which is used to define the path-integral introduced in subsection \ref{202102131116}.
\end{remark}

\begin{remark}
Recall that the source category of a usual Brown functor is the homotopy category $\mathsf{Ho} ( \mathsf{CW}_\ast )$ whereas we consider only finite CW-spaces and refine the definition by the dimension of CW-spaces.
It is formally possible to extend the definition to infinite CW-spaces, but we do not have examples from which we can construct TQFT's.
\end{remark}

\begin{prop}
\label{202101271327}
Let $f : K \to L$, $g: L \to T$ be morphisms in $\mathsf{CW}^\mathsf{fin}_{\ast, \leq r}$.
Let $K \stackrel{f}{\to} L \stackrel{g}{\to} T$ be a mapping cone sequence, i.e. there is a homotopy equivalence $T \simeq C(f)$ under which $g$ is the canonical map.
Then it induces an exact sequence,
\begin{align}
E ( K ) \stackrel{E(f)}{\to} E ( L ) \stackrel{E(g)}{\to} E ( T ) .
\end{align}
\end{prop}
\begin{proof}
It is easy to see that the diagram (\ref{202102021106}) is approximated by a triad of spaces in (\ref{21022211255}) in the sense of Definition \ref{201912190819} where $C(K)$ is the cone of $K$ and $M(f)$ is the mapping cylinder of $f$.
\begin{equation}
\label{202102021106}
\begin{tikzcd}
L \ar[r, "g"] & T \\
K \ar[r] \ar[u, "f"] & \ast \ar[u] 
\end{tikzcd}
\end{equation}

\begin{equation}
\label{21022211255}
\begin{tikzcd}
M(f) \ar[r, hookrightarrow] & C(f) \\
K \ar[r, hookrightarrow] \ar[u, hookrightarrow] & C(K) \ar[u, hookrightarrow] 
\end{tikzcd}
\end{equation}
Hence, the induced chain complex $E(K) \to E(L) \oplus E( \ast ) \to E( T ) $ is exact but $E(\ast) \cong 0$ so that our claim is proved.
\end{proof}

\begin{remark}
The CW-spaces $M(f)$, $C(K)$, $C(f)$ could be $(r+1)$-dimensional whereas the domain of $E$ contains $r$-dimensional ones at most.
It is the reason we require the Mayer-Vietoris axiom for diagrams {\it approximated by a triad of spaces} in Definition \ref{202005211445}.
\end{remark}

\begin{Example}
\label{202102131849}
A homotopy commutative H-group $X$ which is not necessarily a CW-space induces an $\mathsf{Ab}^\mathsf{op}$-valued Brown functor $F_X$ as follows.
We regard the unit of $X$ as a basepoint.
For a pointed CW-space $K$, let $F_X ( K )$ be the homotopy set of pointed maps from $K$ to $X$.
The set $F_X ( K )$ naturally becomes an abelian group due to the H-group structure on $X$.
Furthermore, $F_X$ is a $\mathsf{Ab}^\mathsf{op}$-valued $\infty$-dimensional Brown functor in the sense of Definition \ref{202005211445}.
For example, if $X$ is a twice-iterated loop space $\Omega^2 W$ of a pointed space $W$, then we obtain such a Brown functor.

For some finite conditions on $X$, we obtain an $(\mathsf{Ab}^\mathsf{fin})^\mathsf{op}$-valued Brown functor $F^\natural_X$.
Let $d \in \mathbb{N} \cup \{ \infty \}$.
Suppose that the $r$-th homotopy group $\pi_r ( X)$ is finite for $r \leq d$.
Then $F_X (K)$ is a finite group for $\dim K \leq d$.
Especially it induces an $(\mathsf{Ab}^\mathsf{fin})^\mathsf{op}$-valued $d$-dimensional Brown functor $F^\natural_X$.
\end{Example}

\begin{Example}
\label{202102131904}
We are interested in $\mathsf{Hopf}^\mathsf{bc}_k$-valued Brown functors.
For a field $k$, we have a natural way to obtain a $\mathsf{Hopf}^\mathsf{bc}_k$-valued Brown functor from an $\mathsf{Ab}$-valued Brown functor (e.g. consider a generalized homology theory).
By the group Hopf algebra functor $k(-) : \mathsf{Ab} \to \mathsf{Hopf}^\mathsf{bc}_k$, the category $\mathsf{Ab}$ is an abelian subcategory of $\mathsf{Hopf}^\mathsf{bc}_k$.
Hence, an $\mathsf{Ab}$-valued Brown functor $D$ induces a $\mathsf{Hopf}^\mathsf{bc}_k$-valued Brown functor $E$ by $E(K) = k D(K)$.
Moreover if $D(K)$ is finite and the characteristic of $k$ is coprime to the order of $D(K)$ for any $K$, then $E$ is a $\mathsf{Hopf}^\mathsf{bc,vol}_k$-valued Brown functor.

Analogously, there is a natural way to obtain a $\mathsf{Hopf}^\mathsf{bc}_k$-valued Brown functor from an $(\mathsf{Ab}^\mathsf{fin})^\mathsf{op}$-valued Brown functor (e.g. consider Example \ref{202102131849} or a generalized cohomology theory).
Note that $(\mathsf{Ab}^\mathsf{fin})^\mathsf{op}$ is an abelian subcategory of $\mathsf{Hopf}^\mathsf{bc}_k$ via the function Hopf algebra functor $k^{(-)} : (\mathsf{Ab}^\mathsf{fin})^\mathsf{op} \to \mathsf{Hopf}^\mathsf{bc}_k$.
Let $D$ be an $(\mathsf{Ab}^\mathsf{fin})^\mathsf{op}$-valued Brown functor such that the order of $D(K)$ is coprime to the characteristic of the field $k$ for any $K$.
Then it induces a $\mathsf{Hopf}^\mathsf{bc,vol}_k$-valued Brown functor $E$ by $E(K) = k^{D(K)}$.
\end{Example}

\begin{Example}
\label{202102230954}
Let $A$ be a bicommutative Hopf algebra.
For a pointed finite CW-space $K$ with a basepoint $\ast$, let $\pi_0 (K , \ast)$ be the set of components of $K$ except that of $\ast$.
Note that the assignment of $\pi_0 (K , \ast)$ to $K$ gives a functor.
It induces a functor $E : \mathsf{Ho} ( \mathsf{CW}^\mathsf{fin}_\ast ) \to \mathsf{Hopf}^\mathsf{bc}_k$ given by $E(K) = \otimes_{\pi_0 (K , \ast)} A$.
In particular, if $K$ is connected, then $E(K) = k$, i.e. the trivial Hopf algebra.
It is easy to see that $E$ is a $\mathsf{Hopf}^\mathsf{bc}_k$-valued Brown functor.
Moreover if $A$ has a finite volume, then $E$ is a $\mathsf{Hopf}^\mathsf{bc,vol}_k$-valued.
This example is generalized to Example \ref{202102111047}.
\end{Example}

\begin{Defn}
\label{202102231052}
A Hopf algebra $A$ is 
{\it induced by a group} if there exists a group $G$ such that $A \cong kG$ or a finite group $H$ such that $A \cong k^H$.
A $\mathsf{Hopf}^\mathsf{bc}_k$-valued ($d$-dimensional) Brown functor $E$ {\it induced by groups} if $E(K)$ is induced by a group for any $K$.
\end{Defn}

\begin{remark}
If $k$ is algebraically closed and the characteristic is zero (e.g. $k= \mathbb{C}$), then any finite-dimensional bicommutative Hopf algebra is bisemisimple by Corollary \ref{202006091350} so that it is a group Hopf algebra and simultaneously a function Hopf algebra.
\end{remark}

\begin{remark}
The $\mathsf{Hopf}^\mathsf{bc}_k$-valued Brown functors in Example \ref{202102131904} are induced by groups.
On the other hand, the $\mathsf{Hopf}^\mathsf{bc}_k$-valued Brown functor in Example \ref{202102230954} is induced by groups if and only if $A$ is induced by a group.
An example of $A$ which is not induced by a group is given in \cite{kim2020homology}.
It is possible to construct a $\mathsf{Hopf}^\mathsf{bc}_k$-valued Brown functor which is not induced by groups from tensor product of Brown functors which satisfy some weaker conditions.
See Example \ref{202102231114}, \ref{202102231147}.
\end{remark}

\begin{Example}
\label{202102231114}
If finite-dimensional Hopf algebras $A,B$ satisfy $|\mathrm{GE} (A)| < \dim_k A$ and $|\mathrm{GE} (B)| < \dim_k B$ where $\mathrm{GE}$ is the set of  group-like elements, then $C= A \otimes B^\vee$ is not induced by a group.
In fact, $|\mathrm{GE} (C)| = |\mathrm{GE} (A)| \cdot |\mathrm{GE} (B^\vee)| < \dim_k A \cdot \dim_k B = \dim_k C$, and similarly $|\mathrm{GE} (C^\vee)| < \dim_k C$.
We make use of this observation.
For example, take the Thom spectrum of the orthogonal group $\mathrm{MO}$ which induces a generalized homology theory $\widetilde{\mathrm{MO}}_\bullet$ and cohomology theory $\widetilde{\mathrm{MO}}^\bullet$.
Recall that $\widetilde{\mathrm{MO}}_q(K)$ and $\widetilde{\mathrm{MO}}^q(K)$ are finite for any $K$.
Let $E$ be a $\mathsf{Hopf}^\mathsf{bc}_k$-valued Brown functor given by $E(K) = k \widetilde{\mathrm{MO}}_2(K) \otimes k^{\widetilde{\mathrm{MO}}^{-2} (K)}$ for $k= \mathbb{F}_2$.
It is not induced by groups since $|\mathrm{GE} ( k^{\widetilde{\mathrm{MO}}_2(S^0)} )| = |\mathrm{GE} (k^{\widetilde{\mathrm{MO}}^{-2} (S^0)})| = |\mathrm{GE} ( k^{\mathbb{Z}/2} )| = 1 < 2$ due to $\widetilde{\mathrm{MO}}_2(S^0) \cong \widetilde{\mathrm{MO}}^{-2} (S^0) \cong \mathbb{Z}/2$. 
\end{Example}

\begin{Example}
\label{202102231147}
We give an example analogous to Example \ref{202102231114}.
Let $k$ be a field where the order of $\{ x \in k ~;~ x^{24} = 1\}$ is lower than $24$ (e.g. $\mathbb{R}$, $\mathbb{Q}$, a finite field with order lower than $24$).
Consider the generalized homology theory obtained from the sphere spectrum, $\widetilde{\pi}^s_\bullet$.
It is known that $\widetilde{\pi}^s_3 (S^0) \cong \mathbb{Z}/24$.
Hence the induced $\mathsf{Hopf}^\mathsf{bc,vol}_k$-valued Brown functor $E(K) = k \widetilde{\pi}^s_3(K) \otimes k^{\widetilde{\pi}^{-3}_s(K)}$ is not induced by groups since $\mathrm{GE} ( k^{\mathbb{Z}/24} ) \cong \mathrm{Hom} ( \mathbb{Z}/24 , k^\ast )$ and its order is lower than $24$.
\end{Example}

Given a Brown functor as above, one could naturally construct another Brown functor as follows.

\begin{Defn}
\label{202102231038}
Let $X$ be a pointed finite CW-space.
For a pointed finite CW-space $K$, we define $W_X ( K ) \stackrel{\mathrm{def.}}{=} K \wedge X$.
\end{Defn}

\begin{Example}
If $E$ is an $\mathcal{A}$-valued Brown functor, then the pullback $W^\ast_X E$, i.e. $(W^\ast_X E) (K) = E ( K \wedge X)$, is an $\mathcal{A}$-valued Brown functor.
If $\dim X = r$ and $E$ is $d$-dimensional, then $W^\ast_X E$ is $(d-r)$-dimensional.
\end{Example}

In the remaining part of this subsection, we give a way to obtain a (possibly empty) family of $\mathsf{Hopf}^\mathsf{bc,vol}_k$-valued Brown functors from a $\mathsf{Hopf}^\mathsf{bc}_k$-valued reduced homology theory.
A $\mathsf{Hopf}^\mathsf{bc}_k$-valued reduced homology theory is a sequence of functors from $\mathsf{Ho} ( \mathsf{CW}^\mathsf{fin}_\ast)$ to $\mathsf{Hopf}^\mathsf{bc}_k$ equipped with the suspension isomorphism, which satisfies the Eilenberg-Steenrod axioms except for the dimension axiom \cite{kim2020homology}.
Let $\widetilde{E}_\bullet$ be a $\mathsf{Hopf}^\mathsf{bc}_k$-valued reduced homology theory.

\begin{Example}
\label{202102040954}
Let $G$ be a finite group whose order is coprime to the characteristic of $k$ and $\widetilde{H}_\bullet ( - ; G)$ be the reduced ordinary homology theory with coefficients in $G$.
Consider a $\mathsf{Hopf}^\mathsf{bc}_k$-valued reduced homology theory $\widetilde{E}_\bullet$ given by $\widetilde{E}_q ( K ) = k \widetilde{H}_q ( K ; G)$, i.e. the group Hopf algebra of $\widetilde{H}_q ( K ; G)$.
It gives a $\mathsf{Hopf}^\mathsf{bc}_k$-valued homology theory since the group Hopf algebra functor from $\mathsf{Ab}^\mathsf{fin}$ to $\mathsf{Hopf}^\mathsf{bc}_k$ is an exact functor.

We have a similar example derived from function Hopf algebras.
Let $\widetilde{H}^\bullet ( - ; G)$ be the reduced ordinary cohomology theory with coefficients in $G$.
Then $\widetilde{E}_q ( K ) = k^{\widetilde{H}^q ( K ; G)}$, i.e. the function Hopf algebra of $\widetilde{H}^q ( K ; G)$, gives a $\mathsf{Hopf}^\mathsf{bc}_k$-valued reduced homology theory.
\end{Example}

\begin{Example}
\label{202102111047}
Let $A$ be a bicommutative Hopf algebra with a finite volume and $\widetilde{E}_\bullet = \widetilde{H}_\bullet ( - ; A)$ be the reduced ordinary homology theory with coefficients in $A$.
This example contains Example \ref{202102040954} up to an isomorphism.
In fact, we have isomorphisms of homology theories $\widetilde{H}_\bullet ( - ; k G ) \cong k \widetilde{H}_\bullet ( - ; G)$ and $\widetilde{H}_\bullet ( - ; k^G ) \cong k^{\widetilde{H}^\bullet ( - ; G)}$.
Note that $\widetilde{E}_\bullet$ is induced by groups if and only if $A$ is induced by a group (see Definition \ref{202102231052}).
\end{Example}

\begin{Example}
\label{202102111049}
The construction in Example \ref{202102040954} is applied to any generalized (co)homology theories.
Let $\widetilde{D}_\bullet$ be a generalized reduced homology theory.
The group Hopf algebra functor with coefficients in $k$ induces a $\mathsf{Hopf}^\mathsf{bc}_k$-valued homology theory by $\widetilde{E}_q (K) = k \widetilde{D}_q (K)$.
In a parallel way, a generalized reduced cohomology theory $\widetilde{D}^\bullet$ such that $\widetilde{D}^q (K)$ are finite induces a $\mathsf{Hopf}^\mathsf{bc}_k$-valued homology theory by $\widetilde{E}_q (K) = k^{\widetilde{D}^q (K)}$.
\end{Example}

\begin{Example}
\label{202102111055}
An exponential functor \cite{touze} gives a class of examples.
Let $R$ be another field.
Consider a bicommutative Hopf algebra $A$ over $k$ with an $R$-action $\alpha$.
Then there exists an assignment of a bicommutative Hopf algebra $(A , \alpha)^V$ to a finite-dimensional vector space $V$ over $R$.
It induces an assignment of a $\mathsf{Hopf}^\mathsf{bc}_k$-valued homology theory $\widetilde{E}_\bullet = (A,\alpha)^{\widetilde{D}_\bullet}$ to a $\mathsf{Vec}^\mathsf{fin}_R$-valued homology theory $\widetilde{D}_\bullet$.
Note that a $\mathsf{Vec}^\mathsf{fin}_R$-valued homology theory is nothing but a generalized homology theory such that $\widetilde{D}_q ( K )$ are finite-dimensional vector spaces over $R$.
A typical example of $\widetilde{D}_\bullet$ is given by the MO-theory which is a $\mathsf{Vec}^\mathsf{fin}_R$-valued homology theory for $R= \mathbb{F}_2$.
\end{Example}

\begin{Example}
\label{202102221028}
Moreover if $A$ is not induced by a group, then so the assigned Hopf algebra $(A,\alpha)^{\widetilde{D}_q (K)}$ is for any $K$.
For example, consider $k= \mathbb{F}_3$ and $R= \mathbb{F}_5$.
There exists a finite-dimensional bisemisimple bicommutative Hopf algebra $A$ with a $\mathbb{F}_5$-action which is not induced by a group \cite{kim2020homology}.
Hence, if $\widetilde{D}_\bullet$ is a $\mathsf{Vec}^\mathsf{fin}_{\mathbb{F}_5}$-valued homology theory, then $\widetilde{E}_\bullet = (A,\alpha)^{\widetilde{D}_\bullet}$ is a $\mathsf{Hopf}^\mathsf{bc,vol}_{\mathbb{F}_3}$-valued homology theory which is nontrivial.
Such a homology theory $\widetilde{D}_\bullet$ could be induced by a smash product $P \wedge H \mathbb{F}_5$ for some spectrum $P$ where $H \mathbb{F}_5$ is the Eilenberg-Maclane spectrum of the field $\mathbb{F}_5$.
\end{Example}

\begin{prop}
\label{202002221837}
For $-\infty \leq q_0 \leq q_1 \leq \infty$, the following conditions are equivalent :
\begin{enumerate}
\item
Let $q$ be an integer such that $q_0 \leq q \leq q_1$.
For any pointed finite CW-space $K$ such that $\dim  K  \leq (q - q_0)$, the Hopf algebra $\widetilde{E}_q ( K )$ has a finite volume.
In other words, the restriction $\widetilde{E}_q: \mathsf{Ho} ( \mathsf{CW}^\mathsf{fin}_{\ast, \leq (q - q_0)} ) \to  \mathsf{Hopf}^\mathsf{bc}_k$ factors through $\mathsf{Hopf}^\mathsf{bc,vol}_k$.
\item
The $q$-th coefficient $\widetilde{E}_q ( S^0 )$ has a finite volume for any integer $q$ such that $q_0 \leq q \leq q_1$.
Here, $S^0$ denotes the pointed 0-dimensional sphere.
\item
Let $r$ be any integer such that $0 \leq r \leq (q_1 - q_0)$.
If $q$ is an integer such that $q_0 + r  \leq q \leq q_1$, then the Hopf algebra $\widetilde{E}_q ( K )$ has a finite volume for any pointed $r$-dimensional finite CW-space $K$.
\end{enumerate}
\end{prop}
\begin{proof}
(1) obviously implies (2).

We prove (3) from (2).
If $r=0$, then $K$ is 0-dimensional so that by (2), the Hopf algebra $\widetilde{E}_q ( K )$ has a finite volume for $q_0  \leq q \leq q_1$.
Hence (3) holds for $r = 0$.
Let $r$ be an integer such that $0 \leq r < (q_1 - q_0)$.
Suppose that if $q$ is an integer such that $q_0 + r  < q \leq q_1$, then the Hopf algebra $\widetilde{E}_q ( K )$ has a finite volume for a pointed $r$-dimensional finite CW-space $K$.
Let $L$ be a $(r+1)$-dimensional finite CW-complex.
Let $q$ be an integer such that $q_0 + r + 1 \leq q \leq q_1$.
Consider the long exact sequence associated with the pair $(L , L^{(r)})$ where $L^{(r)}$ is the $r$-skeleton of $L$.
\begin{align}
\widetilde{E}_{q+1} ( L / L^{(r)} )
\to
\widetilde{E}_q (L^{(r)} )
\to
\widetilde{E}_q ( L ) 
\to
\widetilde{E}_q ( L / L^{(r)} )
\to
\widetilde{E}_{q-1} (L^{(r)} )
\end{align}
By the assumption, the Hopf algebras $\widetilde{E}_q ( L^{(r)} ), \widetilde{E}_{q-1}(L^{(r)})$ have a finite volume.
Moreover the quotient complex $L/L^{(r)}$ is homeomorphic to a finite bouquet $\bigvee S^{r+1}$ of the pointed $(r+1)$-dimensional spheres.
From the isomorphism $\widetilde{E}_{q+1} ( \bigvee S^{r+1} ) \cong \widetilde{E}_q ( \bigvee S^{r} )$ and $\widetilde{E}_q ( \bigvee S^{r+1} ) \cong \widetilde{E}_{q-1} ( \bigvee S^{r} )$ and the assumption, the Hopf algebras $\widetilde{E}_{q+1} ( \bigvee S^{r+1} )$ and $\widetilde{E}_q ( \bigvee S^{r+1} )$ have a finite volume.
The Hopf algebra $\widetilde{E}_q ( L)$ has a finite volume since $\mathsf{Hopf}^\mathsf{bc,vol}_k \subset \mathsf{Hopf}^\mathsf{bc}_k$ is closed under short exact sequences.
It proves (3).

We prove (1) from (3).
Let $q$ be an integer such that $q_0 \leq q \leq q_1$.
Let $K$ be a pointed finite CW-space with $\dim K \leq (q - q_0)$.
Put $r = \dim K$.
Since $0 \leq r \leq (q_1 - q_0)$ and $q_0 + r \leq q \leq q_1$ by definitions, the Hopf algebra $\widetilde{E}_q (K)$ has a finite volume by (3).
It completes the proof.
\end{proof}

\begin{Defn}
\label{202001202100}
Denote by $\Gamma ( \widetilde{E}_\bullet )$ the set of integers $q \in \mathbb{Z}$ such that the $q$-th coefficient $\widetilde{E}_q ( S^0 )$ has a finite volume.
For $q \in \Gamma ( \widetilde{E}_\bullet )$, 
we define $d( \widetilde{E}_\bullet ; q) \stackrel{\mathrm{def.}}{=} (q- m(\widetilde{E}_\bullet ; q)) \geq 0$ where
\begin{align}
m(\widetilde{E}_\bullet ; q) \stackrel{\mathrm{def.}}{=} \inf \{ r \in \Gamma ( \widetilde{E}_\bullet ) ~;~ r \leq r^{\prime} \leq q \Rightarrow  r^{\prime} \in \Gamma (\widetilde{E}_\bullet) \} \geq - \infty .
\end{align}
\end{Defn}

\begin{Example}
For both of $\widetilde{E}_\bullet$'s in Example \ref{202102040954}, we have $\Gamma ( \widetilde{E}_\bullet ) = \mathbb{Z}$ and $d( \widetilde{E}_\bullet ; q ) = \infty$.
\end{Example}

\begin{Example}
For the ordinary homology theory in Example \ref{202102111047}, we have $\Gamma ( \widetilde{E}_\bullet ) = \mathbb{Z}$ and $d( \widetilde{E}_\bullet ; q ) = \infty$.
\end{Example}

\begin{Example}
For $\widetilde{E}_\bullet = k \widetilde{D}_\bullet$ in Example \ref{202102111049}, the set $\Gamma ( \widetilde{E}_\bullet )$ consists of $q \in \mathbb{Z}$ such that $\widetilde{D}_q ( S^0 )$ is finite and its order is coprime to the characteristic of $k$.
For example, let $k$ be a field with characteristic zero.
Consider the generalized reduced homology theory $\widetilde{D}_\bullet = \widetilde{\pi}^s_\bullet$ induced by the sphere spectrum, i.e. the stable homotopy theory.
Since the $q$-th coefficient $\widetilde{\pi}^s_q ( S^0)$ is finite if $q \neq 0$, the obtained $\mathsf{Hopf}^\mathsf{bc}_k$-valued homology theory $\widetilde{E}_\bullet = k \widetilde{D}_\bullet$ satisfies $\Gamma ( \widetilde{E}_\bullet ) = \mathbb{Z} \backslash \{ 0 \}$.
\begin{align}
d ( \widetilde{E}_\bullet ; q ) = 
\begin{cases}
(q-1) & \text{$q > 0$,} \\
\infty & \text{$q < 0$.}
\end{cases}
\end{align}
\end{Example}

\begin{Example}
Recall Example \ref{202102111055}.
If $A$ has a finite volume, then we have $\Gamma ( \widetilde{E}_\bullet ) = \mathbb{Z}$ and $d ( \widetilde{E}_\bullet ; q ) = \infty$.
In fact, for any $q$ if $\widetilde{D}_q ( S^0 ) \cong R^{\oplus n}$ for some $n$, then it is easy to verify that $(A, \alpha)^{\widetilde{D}_q ( S^0 )} \cong (A, \alpha)^{R^{\oplus n}} \cong A^{\otimes n}$ has a finite volume.
\end{Example}

\begin{Corollary}
\label{201910242156}
For $q \in \Gamma (\widetilde{E}_\bullet)$, the restriction $\widetilde{E}_q : \mathsf{Ho} ( \mathsf{CW}^\mathsf{fin}_{\ast,\leq d} ) \to \mathsf{Hopf}^\mathsf{bc}_k$ factors through $\mathsf{Hopf}^\mathsf{bc,vol}_k$ where $d = d( \widetilde{E}_\bullet ; q)$.
\end{Corollary}
\begin{proof}
It is immediate from Proposition \ref{202002221837}.
\end{proof}

\begin{Defn}
\label{202102142016}
Let $d = d( \widetilde{E}_\bullet ; q)$.
We define a $d$-dimensional $\mathsf{Hopf}^\mathsf{bc,vol}_k$-valued Brown functor $\widetilde{E}^\natural_q : \mathsf{Ho} ( \mathsf{CW}^\mathsf{fin}_{\ast,\leq d} ) \to \mathsf{Hopf}^\mathsf{bc,vol}_k$ by the induced functor in Corollary \ref{201910242156}.
\end{Defn}

\subsection{Cospanical and spanical extensions}
\label{202102021122}

In this subsection, we give the main theorems in \cite{kim2020extension}.
First we give a review about a cospan category of pointed finite CW-spaces denoted by $\mathsf{Cosp}^\simeq ( \mathsf{CW}^\mathsf{fin}_\ast )$, a cospan category of $\mathcal{A}$ denoted by $\mathsf{Cosp}^\approx ( \mathcal{A} )$ and a span category of $\mathcal{A}$ denoted by $\mathsf{Sp}^\approx ( \mathcal{A} )$.

\subsubsection{A cospan category of pointed CW-spaces}

Let $d \in \mathbb{N} \cup \{ \infty \}$.
Let $K_0, K_1$ be pointed finite CW-spaces with $\dim K_0, \dim K_1 \leq (d-1)$.
Consider cospans of pointed finite CW-spaces $\Lambda = ( K_0 \stackrel{f_0}{\to} L \stackrel{f_1}{\leftarrow} K_1 )$ and $\Lambda^\prime = ( K_0 \stackrel{f^\prime_0}{\to} L^\prime \stackrel{f^\prime_1}{\leftarrow} K_1 )$ with $\dim L , \dim L^\prime \leq d$.
The cospans $\Lambda , \Lambda^\prime$ are homotopy equivalent if there exists a homotopy equivalence $h : L \to L^\prime$ such that $f^\prime_0 \circ h \simeq f_0$ and $f^\prime_1 \circ h \simeq f_1$.
The cospan category $\mathsf{Cosp}^\simeq_{\leq d} ( \mathsf{CW}^\mathsf{fin}_\ast )$ of pointed finite CW-spaces has pointed finite CW-spaces $K$ with $\dim K \leq (d-1)$ as objects and the homotopy equivalence classes of cospans as morphisms.
The composition is given by gluing some mapping cylinders, which is an analogue of the composition of cobordisms.
Then $\mathsf{Cosp}^\simeq_{\leq d} ( \mathsf{CW}^\mathsf{fin}_\ast )$ is a dagger symmetric monoidal category where the monoidal structure is induced by the wedge sum and the dagger structure is induced by the reversing of cospans.

\begin{remark}
The category $\mathsf{Cosp}^\simeq_{\leq d} ( \mathsf{CW}^\mathsf{fin}_\ast )$ is an analogue of cobordism categories.
Nevertheless, there is a big difference between them.
In fact, $(\dim L - \dim K_0)$ and $(\dim L - \dim K_1)$ are not necessarily $1$ or positive.
\end{remark}

\subsubsection{A (co)span category of $\mathcal{A}$}

Let $\mathcal{A}$ be a small abelian category.
We have an equivalence relation $\approx$ of cospans in $\mathcal{A}$ which is generated by a preorder $\preceq$ of cospans.

\begin{Defn}
\label{201912112309}
We define a preorder $\preceq$ of cospans in $\mathcal{A}$.
For cospans $\Lambda = \left( A_0 \stackrel{f_0}{\to} B \stackrel{f_1}{\leftarrow} A_1 \right)$ and $\Lambda^\prime = \left( A^\prime_0 \stackrel{f^\prime_0}{\to} B^\prime \stackrel{f^\prime_1}{\leftarrow} A^\prime_1 \right)$ in $\mathcal{A}$, we define $\Lambda \preceq \Lambda^\prime$ if $A_0 = A^\prime_0$, $A_1 = A^\prime_1$ and there exists a {\it monomorphism} $g : B \to B^\prime$ in $\mathcal{A}$ such that $g \circ f_0 = f^\prime_0$ and $g \circ f_1 = f^\prime_1$.
\end{Defn}

The cospan category associated with $\mathcal{A}$ and $\approx$ is a dagger symmetric monoidal category $\mathsf{Cosp}^\approx ( \mathcal{A})$ whose morphism set is given by $\approx$-equivalence classes of cospans.
In particular, the composition is given by the push-forward construction and the monoidal structure is given by the biproduct in $\mathcal{A}$.
It is well-defined since the equivalence relation $\approx$ is compatible with the direct sum of cospans and the composition of cospans.

\begin{Example}
\label{202102140944}
There is only one $\approx$-equivalence class from the zero object $0$ to itself.
In fact, $(0 \to 0 \leftarrow 0) \preceq ( 0 \to A \leftarrow 0)$ for any object $A$ so that they are $\approx$-equivalent.
\end{Example}

One can define an analogous preorder and an equivalence relation of {\it spans}, and obtain a span category $\mathsf{Sp}^\approx ( \mathcal{A} )$.
Formally, it is given by $\mathsf{Sp}^\approx ( \mathcal{A} ) = \mathsf{Cosp}^\approx ( \mathcal{A}^\mathsf{op})$.
Note that we have a natural isomorphism of dagger symmetric monoidal categories
\begin{align}
\label{202102071200}
\mathrm{T} : \mathsf{Sp}^\approx ( \mathcal{A} ) \to \mathsf{Cosp}^\approx ( \mathcal{A} ) .
\end{align}
For a span $\Lambda = ( B \stackrel{f}{\leftarrow} A \stackrel{f^\prime}{\to} C )$ in $\mathcal{A}$, there exists a cospan $\Lambda^\prime = ( B \stackrel{g}{\to} D \stackrel{g^\prime}{\leftarrow} C )$ such that the square diagram (\ref{201912190830}) in Definition \ref{202101301425} is exact.
Then the functor $\mathrm{T}$ is given by $\mathrm{T} ( [ \Lambda ] ) \stackrel{\mathrm{def.}}{=} [ \Lambda^\prime]$.

\subsubsection{Cospanical and spanical extensions}

\begin{Defn}
\label{202101251135}
Let $F$ be a symmetric monoidal functor $\mathsf{Ho} ( \mathsf{CW}^\mathsf{fin}_{\ast , \leq (d-1)} ) \to \mathcal{A}$.
{\it A cospanical extension of $F$} is a dagger-preserving symmetric monoidal functor $F^\prime$ satisfying the following commutative diagram.
Here, $\iota_{cosp}$ assigns the equivalence class of a cospan $\left( A \stackrel{f}{\to} B \stackrel{Id_B}{\leftarrow} B \right)$ to a morphism $f : A \to B$.
\begin{equation}
\begin{tikzcd}
\mathsf{Ho} ( \mathsf{CW}_{\ast, \leq (d-1)} ) \ar[d, hookrightarrow]  \ar[r, "F"] & \mathcal{A} \ar[d, "\iota_{cosp}"] \\
\mathsf{Cosp}^\simeq_{\leq d} ( \mathsf{CW}^\mathsf{fin}_\ast ) \ar[r, "F^\prime"] & \mathsf{Cosp}^\approx ( \mathcal{A})
\end{tikzcd}
\end{equation}
Analogously, {\it a spanical extension of $F$} is a dagger-preserving symmetric monoidal functor $F^{\prime\prime}$ satisfying the following commutative diagram.
Here, $\iota_{sp}$ assigns the equivalence class of a span $\left( A \stackrel{Id_A}{\leftarrow} A \stackrel{f}{\to} B \right)$ to a morphism $f : A \to B$.
\begin{equation}
\begin{tikzcd}
\mathsf{Ho} ( \mathsf{CW}_{\ast, \leq (d-1)} ) \ar[d, hookrightarrow]  \ar[r, "F"] & \mathcal{A} \ar[d, "\iota_{sp}"] \\
\mathsf{Cosp}^\simeq_{\leq d} ( \mathsf{CW}^\mathsf{fin}_\ast ) \ar[r, "F^{\prime\prime}"] & \mathsf{Sp}^\approx ( \mathcal{A})
\end{tikzcd}
\end{equation}
\end{Defn}

\begin{remark}
In general, it is not obvious whether the functors $\iota_{cosp}$ is faithful or not.
In the case of $\mathcal{A} = \mathsf{Hopf}^\mathsf{bc,vol}_k$, i.e. the category of bicommutative Hopf algebras with a finite volume, it is faithful since $\iota_{cosp} ( f ) = \iota_{cosp} ( f^\prime )$ implies $\hat{\mathrm{I}}_k (\iota_{cosp} ( f )) = \hat{\mathrm{I}}_k (\iota_{cosp} ( f^\prime ))$ (see Definition \ref{202001241711}) which is equivalent with $f  =f^\prime$ by definitions.
\end{remark}

\begin{Defn}
We denote by $i_d : \mathsf{Ho} ( \mathsf{CW}^\mathsf{fin}_{\ast , \leq d} ) \to \mathsf{Ho} ( \mathsf{CW}^\mathsf{fin}_{\ast , \leq (d+1)} )$ the inclusion functor.
\end{Defn}

\begin{theorem}
\label{202001142323}
For $d \in \mathbb{N} \cup \{ \infty \}$, let $E$ be a $d$-dimensional $\mathcal{A}$-valued Brown functor.
There exists a unique cospanical extension of $E \circ i_{d-1}$.
\end{theorem}
\begin{proof}
We sketch a construction of the cospanical extension.
The details are explained in \cite{kim2020extension}.
We consider the assignment of the induced cospan $E( \Lambda)$ in $\mathcal{A}$ to a cospan $\Lambda = ( K_0 \to L \leftarrow K_1 )$ of pointed finite CW-spaces where $\dim L \leq d$, $\dim K_0, \dim K_1 \leq (d-1)$.
Then one can prove that a homotopy equivalence of $\Lambda \simeq \Lambda^\prime$ implies $E ( \Lambda ) \approx E( \Lambda^\prime )$, and the assignment $[ E ( \Lambda )]$ to $[ \Lambda ]$ preserves the compositions.
\end{proof}

\begin{Defn}
We denote by $\Sigma_d : \mathsf{Ho} ( \mathsf{CW}^\mathsf{fin}_{\ast , \leq d} ) \to \mathsf{Ho} ( \mathsf{CW}^\mathsf{fin}_{\ast , \leq (d+1)} )$ the suspension functor.
\end{Defn}

\begin{Defn}
\label{202101291058}
For a cospan $\Lambda = \left( K_0 \stackrel{f_0}{\to} L \stackrel{f_1}{\leftarrow} K_1 \right)$, we define a span of pointed finite CW-spaces by
\begin{align}
\mathrm{T}\Sigma (\Lambda ) \stackrel{\mathrm{def.}}{=} \left( \Sigma K_0 \stackrel{\tau_{K_0} \circ p_0}{\leftarrow} C (f_0 \vee f_1) \stackrel{p_1}{\to} \Sigma K_1 \right) .
\end{align}
Here, $p_0, p_1$ are the collapsing maps from the mapping cone $C(f_0 \vee f_1)$ to the suspensions $\Sigma K_0$ and $\Sigma K_1$ respectively.
The suspension is given by $\Sigma K = S^1 \wedge K$ and $\tau_K : \Sigma K \to \Sigma K ; [t , x] \mapsto [\bar{t} , x]$ where $\bar{t}$ is the conjugate.
\end{Defn}

\begin{theorem}
\label{202101251044}
For $d \in \mathbb{N} \cup \{ \infty \}$, let $E$ be a $d$-dimensional $\mathcal{A}$-valued Brown functor.
There exists a unique spanical extension of $E \circ \Sigma_{d-1}$.
\end{theorem}
\begin{proof}
We sketch a construction of the spanical extension.
The readers are referred to \cite{kim2020extension} for details.
In fact, it is induced by an assignment of a span $E ( \mathrm{T} \Sigma ( \Lambda ) )$ in $\mathcal{A}$ to a cospan of spaces $\Lambda$.
\end{proof}

For a cospan of CW-spaces $\Lambda$, the induced cospan $E ( \Sigma \Lambda )$ and span $E ( \mathrm{T} \Sigma \Lambda )$ are closely related with each other :
\begin{align}
E ( \Sigma \Lambda ) \approx \mathrm{T} ( E ( \mathrm{T} \Sigma ( \Lambda ) ) ) .
\end{align}
Here, $\mathrm{T}$ in the right hand side is explained above.
In fact, it holds if $\dim L \leq (d-1)$ and $\dim K_0 , \dim K_1 \leq (d-2)$ since the following square diagram is exact (see Definition \ref{202101301425}).
\begin{equation}
\begin{tikzcd}
E( \Sigma K_0) \ar[r, "E( \Sigma f_0)"] & E(\Sigma L) \\
E(C(f_0 \vee f_1)) \ar[r , "E(p_1)"'] \ar[u, "E(\tau_{K_0} \circ p_0)"] &  E( \Sigma K_1)  \ar[u, "E( \Sigma f_1)"']
\end{tikzcd}
\end{equation}
The observation gives the following theorem.

\begin{theorem}
\label{201912312039}
For $d \in \mathbb{N} \cup \{ \infty \}$, let $E$ be a $d$-dimensional $\mathcal{A}$-valued Brown functor.
Let $E^\prime = \Sigma^\ast_{d-1} E$ be the induced $(d-1)$-dimensional Brown functor.
For the cospanical extension $\hat{E}^\prime$ of $i^\ast_{d-2} E^\prime$ and the spanical extension $\check{E}$ of $\Sigma^\ast_{d-1} E$, the following diagram commutes up to a natural isomorphism.
\begin{equation}
\begin{tikzcd}
\mathsf{Cosp}^\simeq_{\leq (d-1)} ( \mathsf{CW}^\mathsf{fin}_\ast ) \ar[r, "\hat{E}^\prime"] \ar[d, hookrightarrow, "j_{d-1}"'] & \mathsf{Cosp}^\approx ( \mathcal{A} ) \\
\mathsf{Cosp}^\simeq_{\leq d} ( \mathsf{CW}^\mathsf{fin}_\ast ) \ar[r, "\check{E}"] & \mathrm{Sp}^\approx ( \mathcal{A} ) \ar[u, "\mathrm{T}"] 
\end{tikzcd}
\end{equation}
\end{theorem}

\section{Construction of TQFT's}

In this section, we give constructions of the path-integral of $\mathsf{Hopf}^\mathsf{bc,vol}_k$-valued Brown functors.
Formally speaking, it is given by a symmetric monoidal {\it projective} functor from a cospan category of CW-spaces to the category $\mathsf{C}_k$.
For the definition of symmetric monoidal projective functor, see the appendix.
The construction is formal and categorical in the sense that we compose (co)spanical extensions (see section \ref{202002211539}) with the integral projective functor for (co)spans introduced in this section.

\subsection{Homotopy-theoretic version of TQFT}
\label{202102171836}

In this subsection, we introduce a terminology {\it homotopy-theoretic version of TQFT} as an analogue of TQFT.
It is distinguished from TQFT in the sense that its source category is a cospan category of CW-spaces, not a cobordism category.
Obviously a homotopy-theoretic version of TQFT induces a TQFT in an appropriate way (see Proposition \ref{202102131241}).
The path-integral of $\mathsf{Hopf}^\mathsf{bc,vol}_k$-valued Brown functors, which is the main subject of this section, gives a homotopy-theoretic version of TQFT.

From now on, let $\mathcal{C}$ be a symmetric monoidal category with the unit object $\mathds{1}$.

\begin{Defn}
Let $d \in \mathbb{N} \cup \{ \infty \}$.
{\it A $\mathcal{C}$-valued homotopy-theoretic version of TQFT (HTQFT) with the supremal dimension $d$} is a symmetric monoidal functor from the cospan category $\mathsf{Cosp}^\simeq_{\leq d} ( \mathsf{CW}^\mathsf{fin}_\ast )$ to $\mathcal{C}$.
{\it A $\mathcal{C}$-valued projective HTQFT} is a symmetric monoidal projective functor from the cospan category $\mathsf{Cosp}^\simeq_{\leq d} ( \mathsf{CW}^\mathsf{fin}_\ast )$ to $\mathcal{C}$.
\end{Defn}

\begin{Example}
\label{202102131244}
The copanical extension of $E \circ i_{d-1}$ in Theorem \ref{202001142323} (the spanical extension of $E \circ \Sigma_{d-1}$ in Theorem \ref{202101251044}, resp.) is a $\mathsf{Cosp}^\approx ( \mathcal{A} )$-valued ($\mathsf{Sp}^\approx ( \mathcal{A} )$-valued, resp.) HTQFT.
\end{Example}

\begin{Example}
The example of our main interest is given by the path-integrals of $\mathsf{Hopf}^\mathsf{bc,vol}_k$-valued Brown functors in the following subsection.
\end{Example}

\begin{Defn}
\label{202102120959}
Let $Z$ be a $\mathcal{C}$-valued projective HTQFT with the supremal dimension $d$.
For a pointed finite CW-space $L$ with $\dim L \leq d$, let $\Lambda = ( \ast \to L \leftarrow \ast )$.
Then $Z ( [ \Lambda ] ) : Z( \ast ) \to Z(\ast)$ is a morphism in $\mathcal{C}$.
By the isomorphism $Z(\ast ) \cong \mathds{1}$, we regard $Z ( [ \Lambda ] )$ as an endomorphism on the unit object $\mathds{1}$.
We define {\it the induced $\mathrm{End}_\mathcal{C} ( \mathds{1} )$-valued homotopy invariant (of pointed finite CW-spaces with the dimension lower than or equal to $d$)} by $Z (L) \stackrel{\mathrm{def.}}{=} Z ( [ \Lambda ] ) \in \mathrm{End}_\mathcal{C} ( \mathds{1} )$ with abuse of notations.
\end{Defn}

\begin{Example}
Consider $\mathcal{C} = \mathsf{Cosp}^\approx ( \mathcal{A} )$ with the unit object $\mathds{1} = 0$, i.e. the zero object of $\mathcal{A}$.
The induced homotopy invariant of Example \ref{202102131244} is trivial in the sense that $Z(L) = Id_\mathds{1}$.
In fact, the endomorphism set $\mathrm{End}_\mathcal{C} ( \mathds{1} )$ consists of only the identity $Id_\mathds{1}$ due to Example \ref{202102140944}.
\end{Example}

\begin{Defn}
\label{202006081111}
Let $d \in \mathbb{N} \amalg \{ \infty \}$.
For $n \in \mathbb{N}$ such that $n \leq d$, let $\mathsf{Cob}_n$ be the $n$-dimensional cobordism category of unoriented smooth compact manifolds.
We define a symmetric monoidal functor $\Phi_{n,d} : \mathsf{Cob}_n \to \mathsf{Cosp}_{\leq d} ( \mathsf{CW}^\mathsf{fin}_\ast )$.
It assigns $M^+ = M \amalg \{ \mathrm{pt} \}$ to a closed $(n-1)$-manifold $M$.
It assigns the homotopy equivalence of the cospan induced by embeddings to an $n$-cobordism $N$ from $M_0$ to $M_1$ :
\begin{align}
\Phi_{n,d} ( [ M_0 \hookrightarrow N \hookleftarrow M_1 ] ) \stackrel{\mathrm{def.}}{=} [ M^+_0 \hookrightarrow N^+ \hookleftarrow M^+_1 ] .
\end{align}
\end{Defn}

\begin{remark}
Definition \ref{202006081111} is obviously well-defined.
Especially, the composition is preserved since boundaries have a collar neighborhood.
\end{remark}

A homotopy-theoretic version of TQFT naturally induces a TQFT via the functor $\Phi_{n,d}$ as follows.

\begin{prop}
\label{202102131241}
For $d \in \mathbb{N} \cup \{ \infty \}$ and a natural number $n \leq d$, a $\mathcal{C}$-valued (projective, resp.) HTQFT with the supremal dimension $d$ induces a $\mathcal{C}$-valued (projective, resp.) TQFT with dimension $n$.
\end{prop}
\begin{proof}
For a (projective, resp.) HTQFT $Z$ with the supremal dimension $d$, the composition $Z \circ \Phi_{n,d}$ is (projective, resp.) TQFT with dimension $n$.
\end{proof}

\subsection{Path-integrals of $\mathsf{Hopf}^\mathsf{bc,vol}_k$-valued Brown functors}
\label{202102131116}

In this subsection, we define a pair of the path-integrals of $\mathsf{Hopf}^\mathsf{bc,vol}_k$-valued Brown functors : the cospanical path-integral and the spanical path-integral.

\subsubsection{Cospanical path-integral}

Before a formal explanation, we give our motivation of the cospanical path-integral.
Given a $\mathsf{Hopf}^\mathsf{bc,vol}_k$-valued Brown functor $E$, we obtain a cospan $( E( K_0 ) \to E( L ) \leftarrow E( K_1 ) )$ in $\mathsf{Hopf}^\mathsf{bc,vol}_k$ from a cospan of CW-spaces $( K_0 \to L \leftarrow K_1 )$.
The integral along the cospan gives a morphism from $E(K_0)$ to $E(K_1)$ in $\mathsf{C}_k$.
The assignment is packed into a functor which could be formally described as follows.

\begin{Defn}
\label{202001241711}
We define {\it the integral projective functor for cospans} $\hat{\mathrm{I}}_k$.
It is a symmetric monoidal projective functor $\hat{\mathrm{I}}_k : \mathsf{Cosp}^\approx ( \mathsf{Hopf}^\mathsf{bc,vol}_k ) \to \mathsf{C}_k$ which is the identity on objects and assigns the integral along $\Lambda$, i.e. $\hat{\mathrm{I}}_k ( [ \Lambda ] ) \stackrel{\mathrm{def.}}{=} \int_\Lambda$, to morphisms $[\Lambda]$.
It is a well-defined projective functor by Theorem \ref{201912302324} and Proposition \ref{202005301216}.
\end{Defn}

\begin{Defn}
{\it The cospanical path-integral of $E$} denoted by $\hat{\mathrm{PI}} (E)$ is a HTQFT with the supremal dimension $d$ defined as follows.
Let $\hat{E}$ be the cospanical extension of $E \circ i_{d-1}$ given by Theorem \ref{202001142323}.
We define a symmetric monoidal projective functor $\hat{\mathrm{PI}} (E) \stackrel{\mathrm{def.}}{=} \hat{\mathrm{I}}_k \circ \hat{E}$.
\begin{align}
\mathsf{Cosp}^\simeq_{\leq d} ( \mathsf{CW}^\mathsf{fin}_\ast ) \stackrel{\hat{E}}{\longrightarrow} \mathsf{Cosp}^\approx ( \mathsf{Hopf}^\mathsf{bc,vol}_k ) \stackrel{\hat{\mathrm{I}}_k}{\longrightarrow} \mathsf{C}_k .
\end{align}
\end{Defn}

\begin{Example}
\label{202102121407}
The $k^\ast$-valued homotopy invariant induced by $\hat{\mathrm{PI}} (E)$ (see Definition \ref{202102120959}) is trivial in the sense that $(\hat{\mathrm{PI}} (E ) ) ( L ) = 1 \in k^\ast$.
In fact, $\hat{E} ( [ \Lambda ] ) = [ k \stackrel{\eta}{\to} E(L) \stackrel{\eta}{\leftarrow} k ]$ (see the proof of Theorem \ref{202001142323}) where $\Lambda = ( \ast \to L \leftarrow \ast )$.
Hence, $\hat{\mathrm{I}}_k ( \hat{E} ( [ \Lambda ] ) ) = \sigma^{E(L)} \circ \eta_{E(L)} = 1$.
\end{Example}

\begin{remark}
By Example \ref{202102121407}, one might wonder that the cospanical path-integral (and the following spanical path-integral) is not interesting as an invariant.
As we discuss in the following sections, the cospanical path-integral possesses an obstruction cocycle.
If one solves the associated coboundary equation under some conditions, then the cospanical path-integral could be improved to a (not projective) functor which yields some nontrivial invariant.
\end{remark}

\begin{remark}
In \cite{freed2009topological} \cite{heuts2014ambidexterity} \cite{monnier2015higher}, gauge fields in DWFQ theory are described by classifying maps.
In this sense, DWFQ theory is a sigma-model whose target space is the classifying space.
Our result gives more examples of possible sigma-models which naturally have a quantization by path-integral.
In fact, a spectrum in algebraic-topological sense plays a role of the target space.
It is well-known that a spectrum induces a generalized cohomology theory of CW-spaces.
Such generalized cohomology theory is constructed by a homotopy set of maps from (the suspension spectrum of) spaces to the spectrum.
The function Hopf algebra of a generalized cohomology theory could induce a $\mathsf{Hopf}^\mathsf{bc}_k$-valued homology theory.
By subsection \ref{202102141031}, we obtain a family of $\mathsf{Hopf}^\mathsf{bc,vol}_k$-valued Brown functors to which the cospanical (or spanical in the following subsubsection) path-integral is applied.
\end{remark}

\begin{Lemma}
\label{202006041011}
The cospanical path-integral of $E$ satisfies the following strictly commutative diagram.
\begin{equation}
\begin{tikzcd}
\mathsf{Ho} (\mathsf{CW}_{\ast , \leq (d-1)}) \ar[d, hookrightarrow, "\iota"] \ar[r, hookrightarrow, "i_{d-1}"] & \mathsf{Ho} (\mathsf{CW}_{\ast , \leq d}) \ar[r, "E"] & \mathsf{Hopf}^\mathsf{bc,vol}_k \ar[d, hookrightarrow] \\
\mathsf{Cosp}^\simeq_{\leq d} ( \mathsf{CW}_\ast ) \ar[rr, "\hat{\mathrm{PI}} ( E)"] & & \mathsf{C}_k
\end{tikzcd}
\end{equation}
\end{Lemma}
\begin{proof}
The commutativity of the diagram follows from Theorem \ref{202001142323}.
\end{proof}

\begin{remark}
\label{202102161038}
Let $K$ be a pointed finite CW-space with $\dim K \leq (d-1)$.
By Lemma \ref{202006041011}, the induced representation of the mapping class group $\mathrm{MCG} ( K)$ coincides with the representation induced by the Brown functor $E$ itself.
Here, $\mathrm{MCG} ( K )$ is the quotient set of pointed homotopy equivalences on $K$ divided by homotopies.
In other words, the induced representation of homotopy equivalences (or homeomorphisms) on each space is {\it not projective} whereas $\hat{\mathrm{PI}} (E)$ is projective in general.
\end{remark}

\subsubsection{Spanical path-integral}

A cospan of CW-spaces naturally induces a span of CW-spaces by Definition \ref{202101291058}.
By using such spans, we could define another type of path-integral of Brown functors which is explained in this subsubsection.

\begin{Defn}
\label{202102071245}
We define {\it the integral projective functor for spans} $\check{\mathrm{I}}_k$.
It is a symmetric monoidal projective functor $\check{\mathrm{I}}_k : \mathsf{Sp}^\approx ( \mathsf{Hopf}^\mathsf{bc,vol}_k ) \to \mathsf{C}_k$ which is the identity on objects and assigns the integral along $\mathrm{V}$, i.e. $\check{\mathrm{I}}_k ( [\mathrm{V} ] ) = \int^\mathrm{V}$, to a morphism $[\mathrm{V}]$.
\end{Defn}

\begin{prop}
\label{202102121414}
We have $\check{\mathrm{I}}_k = \hat{\mathrm{I}}_k \circ \mathrm{T}$.
\end{prop}
\begin{proof}
It is immediate from Proposition \ref{202102071235}.
\end{proof}

\begin{Defn}
{\it The spanical path-integral of $E$} denoted by $\check{\mathrm{PI}} (E)$ is a HTQFT with the supremal dimension $d$ defined as follows.
It is a symmetric monoidal projective functor $\check{\mathrm{PI}} (E) \stackrel{\mathrm{def.}}{=} \check{\mathrm{I}}_k \circ \check{E}$ where $\check{E}$ is the spanical extension of $E \circ \Sigma_{d-1}$ (see Theorem \ref{202101251044}).
\begin{align}
\mathsf{Cosp}^\simeq_{\leq d} ( \mathsf{CW}^\mathsf{fin}_\ast ) \stackrel{\check{E}}{\longrightarrow} \mathsf{Sp}^\approx ( \mathsf{Hopf}^\mathsf{bc,vol}_k ) \stackrel{\check{\mathrm{I}}_k}{\longrightarrow} \mathsf{C}_k .
\end{align}
\end{Defn}

\begin{Example}
The $k^\ast$-valued homotopy invariant induced by $\check{\mathrm{PI}} (E)$ is trivial in the sense that $(\check{\mathrm{PI}} (E ) ) ( L ) = 1 \in k^\ast$.
It is verified in a parallel way with Example \ref{202102121407}.
\end{Example}

Some cospanical path-integral and spanical path-integral are related with each other as follows.

\begin{theorem}
\label{202001141422}
We have a natural isomorphism of symmetric monoidal projective functors in the strong sense,
\begin{align}
\check{\mathrm{PI}} ( i^\ast_{d-1} E ) \cong  \hat{\mathrm{PI}} ( \Sigma^\ast_{d-1} E  ) .
\end{align}
\end{theorem}
\begin{proof}
Since $j^\ast_{d-1} ( \check{\mathrm{PI}} (E ) ) = \check{\mathrm{PI}} ( i^\ast_{d-1} E )$ by definitions, it suffices to prove that $j^\ast_{d-1} ( \check{\mathrm{PI}} ( E ) ) \cong \hat{\mathrm{PI}} ( \Sigma^\ast_{d-1} E  )$.
Let $E^\prime =  i^\ast_{d-1} E$.
Let $\hat{E}^\prime$ be the cospanical extension of $E^\prime \circ i_{d-2}$ and $\check{E}$ be the spanical extension of $E \circ \Sigma_{d-1}$.
Consider the following diagram of functors.
The right triangle commutes due to Proposition \ref{202102121414}.
The left diagram commutes by Theorem \ref{201912312039} with $\mathcal{A} = \mathsf{Hopf}^\mathsf{bc,vol}_k$.
It completes the proof.
\begin{equation}
\begin{tikzcd}
\mathsf{Cosp}^\simeq_{\leq (d-1)} ( \mathsf{CW}^\mathsf{fin}_\ast )
\ar[d, "j_{d-1}"']
\ar[r, "\hat{E}^\prime"]
&
\mathsf{Cosp}^\approx ( \mathsf{Hopf}^\mathsf{bc,vol}_k )
\ar[r, "\hat{\mathrm{I}}_k"]
&
\mathsf{C}_k
\\
\mathsf{Cosp}^\simeq_{\leq d} ( \mathsf{CW}^\mathsf{fin}_\ast )
\ar[r, "\check{E}"']
&
\mathsf{Sp}^\approx ( \mathsf{Hopf}^\mathsf{bc,vol}_k )
\ar[u, "\mathrm{T}"]
\ar[ur, "\check{\mathrm{I}}_k"']
&
\end{tikzcd}
\end{equation}
\end{proof}
\section{Obstruction classes induced by integrals}
\label{202102141029}

In the preceding sections, we deal with some symmetric monoidal projective functors such as integral projective functors and path-integrals.
In this section, we study the induced obstruction cocycles and classes.
In particular, we solve the coboundary equations associated with them under some conditions.

\subsection{Obstruction classes associated with integral projective functors}
\label{202002221749}

In this subsection, we study the obstruction cocycles associated with the integral projective functors.
By appendix \ref{202002070917}, the symmetric monoidal projective functor $\hat{\mathrm{I}}_k$ induces a 2-cocycle $\omega ( \hat{\mathrm{I}}_k )$ and 2-cohomology class $\mathds{O} ( \hat{\mathrm{I}}_k )$ of $\mathsf{Cosp}^\approx ( \mathsf{Hopf}^\mathsf{bc,vol}_k )$ with coefficients in the multiplicative group $k^\ast$.
By Proposition \ref{201912280733}, the class $\mathds{O} ( \hat{\mathrm{I}}_k )$ is an obstruction class for the path-integral projective functor to be improved to {\it a symmetric monoidal functor}.

\begin{prop}
\label{202003201443}
The following conditions are equivalent.
\begin{enumerate}
\item
The obstruction cocycle $\omega ( \hat{\mathrm{I}}_k )$ vanishes.
\item
The obstruction class $\mathds{O} ( \hat{\mathrm{I}}_k )$ vanishes.
\item
For any bicommutative Hopf algebra $A$ with a finite volume, we have $( \dim_k A) \cdot 1 = 1 \in k$.
\end{enumerate}
\end{prop}
\begin{proof}
(1) $\Rightarrow$ (2) : It is obvious.

(2) $\Rightarrow$ (3) : Suppose that $\mathds{O} ( \hat{\mathrm{I}}_k )$ vanishes, i.e. $\mathds{O} ( \hat{\mathrm{I}}_k ) =1$.
By Proposition \ref{201912280733}, there exists a symmetric monoidal functor $F : \mathsf{Cosp}^\approx ( \mathsf{Hopf}^\mathsf{bc,vol}_k ) \to \mathsf{C}_k$ such that $F \cong_\mathrm{proj}  \hat{\mathrm{I}}_k$.
Let $A$ be a bicommutative Hopf algebra with a finite volume.
Note that the cospan $\Lambda$ and its dagger $\Lambda^\dagger$ in Definition \ref{202006081524} gives a self-duality of $A$ in $\mathsf{Cosp}^\approx ( \mathsf{Hopf}^\mathsf{bc,vol}_k )$.
Let $d$ be the dimension of $A$ in the symmetric monoidal category $\mathsf{Cosp}^\approx ( \mathsf{Hopf}^\mathsf{bc,vol}_k )$, i.e. $d = [ \Lambda^\dagger ] \circ [ \Lambda ]$.
Then $d$ is the identity on the unit object $k$ since the endomorphism set of $k$ in $\mathsf{Cosp}^\approx ( \mathsf{Hopf}^\mathsf{bc,vol}_k )$ has only the identity (see Example \ref{202102140944}).
Since $F$ is a symmetric monoidal functor, the dimension of $F(A) = A$ in $\mathsf{C}_k$ is $1 : k \to k$.

(3) $\Rightarrow$ (1) : 
Let $A$ be an arbitrary bicommutative Hopf algebra with a finite volume.
Then $A$ is dualizable and its dimension in $\mathsf{C}_k$ coincides with $\mathrm{vol}^{-1} (A)^{-1}$ by Corollary \ref{202101232004}.
By the assumption, we have $\mathrm{vol}^{-1} (A) = 1$.
By Proposition 11.9 in \cite{kim2021integrals}, we have $\langle \xi \rangle = 1 \in k$ for any homomorphism $\xi : B \to B^\prime$ between bicommutative Hopf algebras with a finite volume.
Therefore the cocycle $\omega ( \hat{\mathrm{I}}_k )$ vanishes by definitions.
\end{proof}

From the above proposition, we show some nontriviality of the second cohomology group of the symmetric monoidal category $\mathsf{Cosp}^\approx ( \mathsf{Hopf}^\mathsf{bc,vol}_k)$.

\begin{Corollary}
\label{202006081511}
Let $p$ be the characteristic of the ground field $k$.
\begin{enumerate}
\item
The obstruction class $\mathds{O} ( \hat{\mathrm{I}}_k )$ vanishes if and only if $p = 2$.
In that case, the cocycle $\omega ( \hat{\mathrm{I}}_k )$ vanishes.
\item
If $p \neq 2$, then the second cohomology group $H^2 ( \mathsf{Cosp}^\approx ( \mathsf{Hopf}^\mathsf{bc,vol}_k ); k^\ast )$ is not trivial.
\end{enumerate}
\end{Corollary}
\begin{proof}
Note that if $p\neq 2$, then there exists a bicommutative Hopf algebra with a finite volume whose dimension in $\mathsf{C}_k$ is not $1 \in k$.
Such examples could be obtained from group Hopf algebras.
If $p = 2$, then the dimension of any bicommutative Hopf algebra with a finite volume is $1\in k$ since the dimension should be invertible in $k$ by Corollary \ref{202006040804}.
It proves the first claim.
By the first claim, the class $\mathds{O} ( \hat{\mathrm{I}}_k ) \neq 1$ if $p \neq 2$.
It proves the second claim.
\end{proof}

In other words, the following coboundary equation associated with the integral projective functor $\hat{I}_k$ is solvable if and only if $p=2$.
\begin{align}
\omega ( \hat{I}_k ) = \delta \theta .
\end{align}
By repeating the discussion analogous to the above one, we obtain the following result.

\begin{Corollary}
Let $p$ be the characteristic of the ground field $k$.
\begin{enumerate}
\item
If $p \neq 0,2$, then the cohomology group $H^2 ( \mathsf{Cosp}^\approx ( \mathsf{Hopf}^\mathsf{bc,vol}_k ); \mathbb{F}^\ast_p )$ is not trivial.
\item
If $p = 0$, then the cohomology group $H^2 ( \mathsf{Cosp}^\approx ( \mathsf{Hopf}^\mathsf{bc,vol}_k ); \mathbb{Q}_{>0} )$ is not trivial.
\end{enumerate}
\end{Corollary}
\begin{proof}
Let $G$ be the multiplicative group $\mathbb{F}^\ast_p$ if $p \neq 0,2$ or $\mathbb{Q}_{>0}$ if $p = 0$.
Then the obstruction cocycle $\omega ( \hat{\mathrm{I}}_k )$ has coefficients in $G$ due to Corollary \ref{202006040804}.
It induces a class $[ \omega ( \hat{\mathrm{I}}_k ) ] \in H^2 ( \mathsf{Cosp}^\approx ( \mathsf{Hopf}^\mathsf{bc,vol}_k ); G )$.
The coefficient extension map assigns $\mathds{O} ( \hat{\mathrm{I}}_k) \in H^2 ( \mathsf{Cosp}^\approx ( \mathsf{Hopf}^\mathsf{bc,vol}_k ); k^\ast )$ to the class $[ \omega ( \hat{\mathrm{I}}_k ) ] \in H^2 ( \mathsf{Cosp}^\approx ( \mathsf{Hopf}^\mathsf{bc,vol}_k ); G )$.
By Corollary \ref{202006081511}, the class $[ \omega ( \hat{\mathrm{I}}_k ) ]$ is nontrivial.
It completes the proof.
\end{proof}

\subsection{Obstruction classes associated with Brown functors}
\label{202002230946}
In this subsection, we define the obstruction classes of the cospanical and spanical path-integrals, and give a basic relation of them.
Throughout this subsection, let $E$ be a $d$-dimensional $\mathsf{Hopf}^\mathsf{bc,vol}_k$-valued Brown functor for $d \in \mathbb{N} \cup \{ \infty \}$.

\begin{Defn}
\label{202003031321}
We define a cocycle $\hat{\omega} ( E ) \stackrel{\mathrm{def.}}{=} \omega ( \hat{\mathrm{PI}} (E) )$ (see Definition \ref{202102161148}) and a cohomology class $\hat{\mathds{O}}(E) \stackrel{\mathrm{def.}}{=} \mathds{O} ( \hat{\mathrm{PI}} (E) )$ (see Proposition \ref{201912252055}) of the symmetric monoidal category $\mathsf{Cosp}^\simeq_{\leq d} ( \mathsf{CW}^\mathsf{fin}_\ast )$ with coefficients in $k^\ast$.
Equivalently, it is given by a pullback $\hat{\omega}(E) = \hat{E}^\ast ( \omega ( \hat{I}_k ) )$ and $\hat{\mathds{O}}(E) = \hat{E}^\ast ( \mathds{O} ( \hat{I}_k ) )$.
Analogously, we define a cocycle $\check{\omega} ( E) \stackrel{\mathrm{def.}}{=} \omega ( \check{\mathrm{PI}} (E) )$ and a cohomology class $\check{\mathds{O}}(E) \stackrel{\mathrm{def.}}{=} \mathds{O} ( \check{\mathrm{PI}} (E) )$ in $H^2 ( \mathsf{Cosp}^\simeq_{\leq d} (\mathsf{CW}^\mathsf{fin}_\ast ) ; k^\ast)$.
In other words, it is given by a pullback $\check{\mathds{O}}(E) = \check{E}^\ast ( \mathds{O} ( \check{I}_k ) )$.
\end{Defn}

\begin{remark}
The cohomology class $\hat{\mathds{O}}(E)$ ($\check{\mathds{O}}(E)$, resp.) is an obstruction class for $\hat{\mathrm{PI}} (E)$ ($\check{\mathrm{PI}} (E)$, resp.) to be lifted to a symmetric monoidal functor.
If we have a solution $\theta$ of the coboundary equation associated with $\hat{\mathrm{PI}} (E)$ ($\check{\mathrm{PI}} (E)$, resp.), then the lift $\theta^{-1} \cdot \hat{\mathrm{PI}} (E)$ ($\theta^{-1} \cdot \check{\mathrm{PI}} (E)$, resp.) is a homotopy-theoretic version of TQFT.
\end{remark}

\begin{Example}
Suppose that the characteristic of $k$ is $2$.
By Corollary \ref{202006081511}, the obstruction cocycles $\hat{\omega} (E)$ and $\check{\omega} (E)$ vanish.
Hence, the path-integrals $\hat{\mathrm{PI}} (E)$ and $\check{\mathrm{PI}} (E)$ are HTQFT's.
\end{Example}

We give a formula of obstruction cocycles (classes) as follows.

\begin{Corollary}
\label{202006041143}
For the induced obstruction cocycles, we have the following equation.
\begin{align}
\check{\omega} ( i^\ast_{d-1} E ) =  \hat{\omega} ( \Sigma^\ast_{d-1} E)  .
\end{align}
In particular, we have $\check{\mathds{O}} ( i^\ast_{d-1} E ) = \hat{\mathds{O}} ( \Sigma^\ast_{d-1} E )$.
\end{Corollary}
\begin{proof}
It is immediate from Theorem \ref{202001141422}.
\end{proof}

There is a typical Brown functor induced by a $\mathsf{Hopf}^\mathsf{bc}_k$-valued homology theory.
See Definition \ref{202102142016}.
In application of Corollary \ref{202006041143}, we obtain the following result.

\begin{Corollary}
Let $\widetilde{E}_\bullet$ be a $\mathsf{Hopf}^\mathsf{bc}_k$-valued reduced homology theory.
Let $q \in \Gamma ( \widetilde{E}_\bullet )$ such that $(q+1) \in \Gamma ( \widetilde{E}_\bullet )$.
Denote by $j_d : \mathsf{Cosp}^\simeq_{\leq d} ( \mathsf{CW}^\mathsf{fin}_\ast ) \to \mathsf{Cosp}^\simeq_{\leq (d+1)} ( \mathsf{CW}^\mathsf{fin}_\ast )$ the inclusion functor where $d = d( \widetilde{E}_\bullet ; q)$.
Then we have a natural isomorphism of symmetric monoidal projective functors in the strong sense,
\begin{align}
j^\ast_d ( \check{\mathrm{PI}} (\widetilde{E}^\natural_{q+1}) ) \cong  \hat{\mathrm{PI}} (\widetilde{E}^\natural_q) .
\end{align}
In particular, we have $j^\ast_d ( \check{\omega} (\widetilde{E}^\natural_{q+1}) ) = \hat{\omega} ( \widetilde{E}^\natural_q) $.
\end{Corollary}
\begin{proof}
The suspension isomorphism gives $\Sigma^\ast_d \widetilde{E}^\natural_{q+1} \cong \widetilde{E}^\natural_q$.
The statement is immediate from Theorem \ref{202001141422}.
\end{proof}

\section{Obstruction classes associated with properly extensible Brown functor}
\label{202102150935}

In the previous sections, we construct some HTQFT's and study basic properties of the associated obstruction classes.
We do not know whether the obstruction classes vanish or not in general.
In this section, we introduce a specific class of Brown functors for which some computation is possible.

\subsection{Inversion formulae}
\label{202102150945}

In this subsection, we study the obstruction cocycles for specific class of Brown functors, called {\it properly extensible} Brown functors.
We derive some formulae to solve the coboundary equations associated with the path-integrals of properly extensible Brown functors.

\begin{Defn}
\label{202102171348}
A $\mathsf{Hopf}^\mathsf{bc,vol}_k$-valued $d$-dimensional Brown functor $E$ is {\it properly extensible} if there exists a $\mathsf{Hopf}^\mathsf{bc}_k$-valued $(d+2)$-dimensional Brown functor $E_+$ equipped with a natural isomorphism $E_+ ( \Sigma K) \cong E(K)$ for $\dim K \leq d$.
We call a distinguished $E_+$ by {\it a proper extension of $E$}.
\end{Defn}

\begin{remark}
Note that the dimension of Brown functor $E_+$ needs to be $(d+2)$, not merely $(d+1)$.
In the proof of Theorem \ref{202001271630}, $E_+ ( \Sigma \partial )$ is well-defined since $E_+$ is $(d+2)$-dimensional where $\partial$ is a map to a CW-space with dimension lower than or equal to $(d+1)$.
It is essential for the claim $\beta ( [ \Lambda^\prime ] , [ \Lambda ] ) = \mathrm{vol}^{-1}  ( \mathrm{Im_H} ( E_+ ( \Sigma \partial ) ) )$.
\end{remark}

\begin{Example}
\label{202102151212}
Let $Y$ be a homotopy commutative H-group and $d \in \mathbb{N} \cup \{ \infty \}$.
Suppose that $\pi_r (Y)$ is finite for $r \leq (d+2)$.
Note that $X = \Omega Y$ is also a homotopy commutative H-group such that $\pi_r (X)$ is finite for $r \leq (d+1)$.
If the characteristic of $k$ is zero, then we obtain a $\mathsf{Hopf}^\mathsf{bc,vol}_k$-valued $d$-dimensional Brown functor $E(K) = k^{F^\natural_X (K)}$ by Example \ref{202102131904}.
Then $E$ is properly extensible.
In fact, $E_+ (K) = k^{F^\natural_Y (K)}$ gives a proper extension of $E$ since we have $F^\natural_X (K) = F^\natural_{\Omega Y} (K) \cong F^\natural_Y ( \Sigma K)$ due to the adjoint relation between the suspension and looping \cite{hatcher2002algebraic}.
\end{Example}

\begin{Example}
\label{202102142024}
Let $\widetilde{E}_\bullet$ be a $\mathsf{Hopf}^\mathsf{bc}_k$-valued homology theory.
For $q \in \Gamma ( \widetilde{E}_\bullet )$, the induced Brown functor $\widetilde{E}^\natural_q$ is properly extensible.
Let $d = d ( \widetilde{E}_\bullet ; q)$.
In fact, the restriction of $\widetilde{E}_{q+1}$ to $\mathsf{Ho} ( \mathsf{CW}^\mathsf{fin}_{\ast , \leq (d+1)} )$ is a proper extension of $\widetilde{E}^\natural_q$ due to the suspension isomorphism.
\end{Example}

\begin{Defn}
\label{202001141418}
Let $E$ be a $\mathsf{Hopf}^\mathsf{bc,vol}_k$-valued $d$-dimensional Brown functor.
We define a 1-cochain $\theta ( E )$ of the symmetric monoidal category $\mathsf{Cosp}^\simeq_{\leq d} ( \mathsf{CW}^\mathsf{fin}_\ast )$ with coefficients in $k^\ast$.
Let $[\Lambda]$ be a morphism in $\mathsf{Cosp}^\simeq_{\leq d} ( \mathsf{CW}^\mathsf{fin}_\ast )$ where $\Lambda = \left( K_0 \stackrel{f_0}{\to} L \stackrel{f_1}{\leftarrow} K_1 \right)$.
Then,
\begin{align}
\label{202002222221}
\left( \theta ( E ) \right)  ( [\Lambda] ) \stackrel{\mathrm{def.}}{=} \mathrm{vol}^{-1} ( E ( C (f_1) ) ) \in k^\ast .
\end{align}
Here, $C (f_1)$ is the mapping cone of the pointed map $f_1$.
Since $C(f_1)$ is a complex with the dimension lower than $d$, the bicommutative Hopf algebra $E ( C (f_1) )$ has a finite volume so that (\ref{202002222221}) is well-defined.
\end{Defn}

\begin{Lemma}
\label{201911131255}
Consider an exact sequence in the abelian category $\mathsf{Hopf}^\mathsf{bc}_k$ :
\begin{align}
C_1 \stackrel{\partial_1}{\to} A_0 \to B_0 \to C_0 \stackrel{\partial_0}{\to} A_{-1}  .
\end{align}
Suppose that $A_0,B_0,C_0$ have a finite volume.
Then the image Hopf algebras $\mathrm{Im_H} ( \partial_0 ), \mathrm{Im_H} ( \partial_1 )$ have a finite volume and we have
\begin{align}
\mathrm{vol}^{-1} ( \mathrm{Im_H} ( \partial_0 ) ) \cdot \mathrm{vol}^{-1} ( \mathrm{Im_H} ( \partial_1 ) )  = \mathrm{vol}^{-1} ( A_0 ) \cdot \mathrm{vol}^{-1} (B_0 )^{-1} \cdot \mathrm{vol}^{-1} (C_0 ) . 
\end{align}
\end{Lemma}
\begin{proof}
The first claim is proved by the fact that $\mathsf{Hopf}^\mathsf{bc,vol}_k$ is an abelian subcategory of $\mathsf{Hopf}^\mathsf{bc}_k$.
Note that the exact sequence induces the following exact sequence :
\begin{align}
k \to \mathrm{Im_H} ( \partial_1 ) \stackrel{\partial_1}{\to} A_0 \to B_0 \to C_0 \stackrel{\partial_0}{\to} \mathrm{Coim_H} ( \partial_0 ) \to k .
\end{align}
Since the inverse volume is a volume on the abelian category $\mathsf{Hopf}^\mathsf{bc,vol}_k$ (see Corollary \ref{202002211513}), we obtain the following equation :
\begin{align}
\mathrm{vol}^{-1} ( \mathrm{Im_H} ( \partial_1 ) ) \cdot \mathrm{vol}^{-1} ( B_0 ) \cdot \mathrm{vol}^{-1} ( \mathrm{Coim_H} ( \partial_0 ) ) = \mathrm{vol}^{-1} ( A_0 ) \cdot \mathrm{vol}^{-1} ( C_0 ) .
\end{align}
By the isomorphism $\mathrm{Coim_H} ( \partial_0 ) \cong \mathrm{Im_H} ( \partial_0 ) )$, the claim is proved.
\end{proof}

\begin{Lemma}
\label{202101301326}
Consider a commutative diagram in the category $\mathsf{Hopf}^\mathsf{bc}_k$.
If both of the upper and lower square diagrams are exact (see Definition \ref{202101301425}), then we have a Hopf isomorphism $\mathrm{Im_H} ( \partial ( \xi^\prime , \xi ) ) \cong \mathrm{Im_H} (  \partial ( \varphi^\prime , \varphi ) )$ (see Definition \ref{202101301421}).
\begin{equation}
\begin{tikzcd}
A_2 \ar[r] & B_2 \\
A_1 \ar[u, "\xi^\prime"]  \ar[r]  & B_1 \ar[u, "\varphi^\prime"] \\
A_0 \ar[u, "\xi"]  \ar[r] & B_0 \ar[u, "\varphi"]
\end{tikzcd}
\end{equation}
\end{Lemma}
\begin{proof}
The commutativity induces the following commutative diagram.
The homomorphism $u$ is an epimorphism and $v$ is a monomorphism since the square diagrams in the statement are exact.
Hence we obtain a natural isomorphism $\mathrm{Im_H} ( \partial ( \xi^\prime , \xi ) ) \cong \mathrm{Im_H} (  \partial ( \varphi^\prime , \varphi ) )$ since $\mathsf{Hopf}^\mathsf{bc}_k$ is an abelian category.
\begin{equation}
\begin{tikzcd}
\mathrm{Ker_H} ( \xi^\prime ) \ar[r, "\partial ( \xi^\prime ~ \xi )"] \ar[d, "u"] & \mathrm{Cok_H} ( \xi ) \ar[d, "v"] \\
\mathrm{Ker_H} ( \varphi^\prime ) \ar[r, "\partial ( \varphi^\prime ~ \varphi )"] & \mathrm{Cok_H} ( \varphi )
\end{tikzcd}
\end{equation}
\end{proof}

The obstruction cocycles associated with the cospanical and spanical path-integrals are inverses to each other up to a coboundary as follows.

\begin{theorem}[Inversion formula 1]
\label{202001271630}
If a $\mathsf{Hopf}^\mathsf{bc,vol}_k$-valued $d$-dimensional Brown functor $E$ is properly extensible, then the coboundary equation associated with $\hat{\mathrm{PI}} ( E ) \otimes \check{\mathrm{PI}} ( E )$ is solvable.
Furthermore, the 1-cochain $\theta ( E )$ gives a canonical solution :
\begin{align}
\hat{\omega} (E)  \cdot  \check{\omega} (E)   =  \delta ( \theta ( E )) .
\end{align}
Especially, we have $\hat{\mathds{O}} ( E ) = \check{\mathds{O}} ( E )^{-1}$.
\end{theorem}
\begin{proof}
Consider composable morphisms $[ \Lambda^\prime ] , [ \Lambda ]$ in $\mathsf{Cosp}^\simeq_{\leq d} ( \mathsf{CW}^\mathsf{fin}_\ast )$ with the representatives,
\begin{align}
\Lambda = \left( K_0 \stackrel{f_0}{\to} L \stackrel{f_1}{\leftarrow} K_1 \right) ,~~~ \Lambda^\prime = \left( K_1 \stackrel{f^\prime_1}{\to} L^\prime \stackrel{f^\prime_2}{\leftarrow} K_2 \right) .
\end{align}
We introduce notations of maps associated with the composition $\Lambda^\prime \circ \Lambda$ and $\mathrm{T}\Sigma ( \Lambda )$, $\mathrm{T}\Sigma ( \Lambda^\prime )$, $\mathrm{T}\Sigma ( \Lambda^\prime \circ \Lambda )$ following Figure \ref{202001141113}, \ref{202001141124} (see Definition \ref{202101291058} for $\mathrm{T}\Sigma$).

\begin{figure}[h]
  \includegraphics[width=7.9cm]{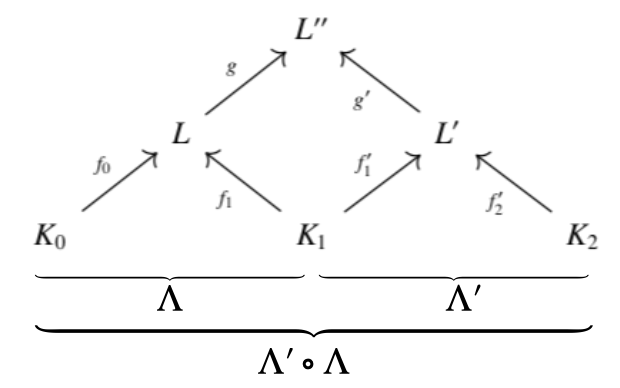}
  \caption{A part of the proof of Theorem \ref{202001271630}}
  \label{202001141113}
\end{figure}

\begin{figure}[h]
  \includegraphics[width=8.3cm]{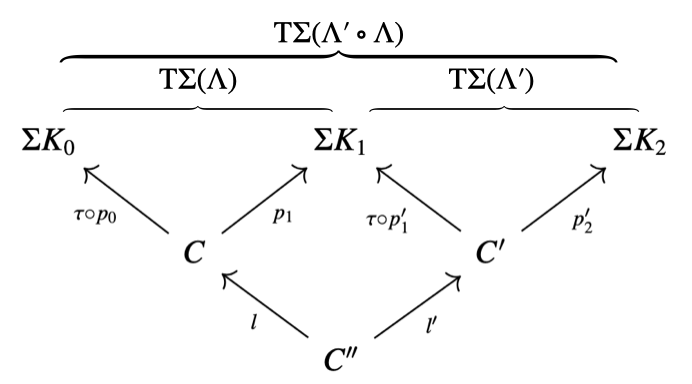}
  \caption{A part of the proof of Theorem \ref{202001271630}}
  \label{202001141124}
\end{figure}

For simplicity, denote by $\alpha = \hat{\omega} (E)$ and $\beta = \check{\omega} (E)$.
Recall the definition, $\alpha  ( [ \Lambda^\prime ] , [ \Lambda ] )  = \langle \partial ( E ( g^\prime ), E ( f^\prime_2 ) ) \rangle $ (see Definition \ref{202101301421} for $\partial ( \xi^\prime , \xi )$).
Choose a proper extension $E_+$ of $E$.
By the definition of $E_+$, we have $\alpha  ( [ \Lambda^\prime ] , [ \Lambda ] )  = \langle \partial ( E_+ ( \Sigma g^\prime ), E_+ ( \Sigma f^\prime_2 ) ) \rangle$.
Let $\partial : C ( g^\prime ) \to \Sigma C(f^\prime_2 )$ be the connecting map associated with the mapping cone sequence $C(f^\prime_2 ) \to C(g^\prime \circ f^\prime_2 ) \to C(g^\prime )$.
Note that the images of $\partial ( E_+ ( \Sigma g^\prime) , E_+ ( \Sigma f^\prime_2 ) )$ and $E_+ ( \partial )$ are canonically isomorphic with each other.
In fact, it is due to the following (horizontal and vertical) exact sequences associated with the mapping cone of $g^\prime$ and $f^\prime_2$.
These follow from Proposition \ref{202101271327}.
Hence we obtain $\alpha  ( [ \Lambda^\prime ] , [ \Lambda ] )   = \mathrm{vol}^{-1} ( \mathrm{Im_H} ( E_+ ( \partial ) ) )$.
\begin{equation}
\begin{tikzcd}
& E_+ ( \Sigma K_2 ) \ar[d, "E_+(\Sigma f^\prime_2)"'] & \\
E_+ ( C ( g^\prime ) ) \ar[r] \ar[dr, "E_+(\partial)"'] & E_+ ( \Sigma L^\prime ) \ar[r, "E_+(\Sigma g^\prime)"'] \ar[d]  & E_+ ( \Sigma L^{\prime\prime} ) \\
& E_+ ( \Sigma ( C ( f^\prime_2 ) ) ) & 
\end{tikzcd}
\end{equation}

On the one hand, we claim that $\beta  ( [ \Lambda^\prime ] , [ \Lambda ] ) = \mathrm{vol}^{-1} ( \mathrm{Im_H} ( E_+ ( \Sigma \partial ) ) )$.
In fact, for $F = \Sigma^\ast E_+ = E_+ \circ \Sigma$, the diagrams in Figure \ref{202001141113}, \ref{202001141124} induce the commutative diagram (\ref{202101301246}).
Note that the suspensions $\Sigma L, \Sigma L^\prime , \Sigma L^{\prime\prime}$ are $(d+1)$-dimensional at most.
We remark that $F$ of them are well-defined since the domain of $E_+$ contains all the pointed finite CW-spaces with $\dim \leq (d+2)$.
\begin{equation}
\label{202101301246}
\begin{tikzcd}
F( \Sigma K_0 ) \ar[r, "F(\Sigma f_0)"] & F( \Sigma L ) \ar[r, "F(\Sigma g)"] & F( \Sigma L^{\prime\prime} ) \\
F ( C ) \ar[u, "F(\tau \circ p_0)"] \ar[r, "F(p_1)"] & F ( \Sigma K_1 ) \ar[u, "F(\Sigma f_1)"] \ar[r, "F(\Sigma f^\prime_1)"] & F ( \Sigma L^\prime ) \ar[u, "F(\Sigma g^\prime)"] \\
F ( C^{\prime\prime} ) \ar[u, "F(l)"] \ar[r, "F(l^\prime)"] & F ( C^\prime ) \ar[u, "F(\tau \circ p^\prime_1)"] \ar[r, "F(p^\prime_2)"] &  F ( \Sigma K_2 ) \ar[u, "F(\Sigma f^\prime_2)"]
\end{tikzcd}
\end{equation}
All the square diagrams appearing in (\ref{202101301246}) are exact since $F$ is a $(d+1)$-dimensional Brown functor.
By iterated applications of Lemma \ref{202101301326}, we obtain
\begin{align}
& \mathrm{Im_H} ( \partial ( F(\tau \circ p_0) , F ( l ) ) ) ,\\
\cong& ~ \mathrm{Im_H} ( \partial ( F( \Sigma f_1) , F(\tau \circ p^\prime_1 ) ) ) ,\\
\cong& ~ \mathrm{Im_H} ( \partial ( F ( \Sigma g^\prime ) , F( \Sigma f^\prime_2 ) ) ) .
\end{align}
Hence, we obtain $\beta  ( [ \Lambda^\prime ] , [ \Lambda ] ) = \mathrm{vol}^{-1} (\mathrm{Im_H} ( \partial ( F ( \Sigma g^\prime ) , F( \Sigma f^\prime_2 ) ) ))$.
In a parallel way with the above computation of $\alpha ( [ \Lambda^\prime ] , [ \Lambda ] )$, we can obtain $\mathrm{vol}^{-1} (\mathrm{Im_H} ( \partial ( F ( \Sigma g^\prime ) , F( \Sigma f^\prime_2 ) ) )) = \mathrm{vol}^{-1} ( \mathrm{Im_H} ( E_+ ( \Sigma \partial ) ) )$.
It proves our claim.

Above all, we obtain $\alpha  ( [ \Lambda^\prime ] , [ \Lambda ] ) \cdot \beta  ( [ \Lambda^\prime ] , [ \Lambda ] ) = \mathrm{vol}^{-1} ( \mathrm{Im_H} ( E_+ ( \partial ) ) ) \cdot \mathrm{vol}^{-1} ( \mathrm{Im_H} ( E_+ ( \Sigma \partial ) ) )$.
By applying Lemma \ref{201911131255} to the induced exact sequence (\ref{202101271337}) which follows from Proposition \ref{202101271327}, one could verify that the product $\alpha  ( [ \Lambda^\prime ] , [ \Lambda ] ) \cdot \beta  ( [ \Lambda^\prime ] , [ \Lambda ] )$ coincides with $\delta ( \theta ( E )) ([ \Lambda^\prime ] , [ \Lambda ] )$.
It completes the proof.

\begin{align}
\label{202101271337}
E_+ (C(g^\prime) )
\stackrel{E_+ ( \partial )}{\to}
E_+ (\Sigma C(f^\prime_2))
\to
E_+ (\Sigma C(g^\prime \circ f^\prime_2 ) )
\to
E_+ (\Sigma C(g^\prime))
\stackrel{E_+ ( \Sigma\partial )}{\to}
E_+ (\Sigma^2 C(f^\prime_2 ) ) .
\end{align}
\end{proof}

\begin{remark}
There is another way to prove $\beta  ( [ \Lambda^\prime ] , [ \Lambda ] ) = \mathrm{vol}^{-1} ( \mathrm{Im_H} ( E_+ ( \Sigma \partial ) ) )$ in the proof of Theorem \ref{202001271630}.
Note that $\beta  ( [ \Lambda^\prime ] , [ \Lambda ] )  = \mathrm{vol}^{-1} ( \mathrm{Im_H} ( E_+ ( \partial^\prime ) ) )$ where $\partial^\prime : C( \tau \circ p_0 ) \to \Sigma C(l)$ is the connecting map associated with the mapping cone sequence $C(l) \to C( \tau \circ p_0 \circ l ) \to C( \tau \circ p_0 )$.
\begin{equation}
\begin{tikzcd}
& E_+ ( \Sigma C^{\prime\prime} ) \ar[d, "E_+(\Sigma l)"'] & \\
E_+ ( C ( \tau \circ p_0 ) ) \ar[r] \ar[dr, "E_+(\partial^\prime)"'] & E_+ ( \Sigma C ) \ar[r, "E_+(\Sigma \tau \circ p_0)"'] \ar[d]  & E_+ ( \Sigma \Sigma K_0 ) \\
& E_+ ( \Sigma ( C ( l ) ) ) & 
\end{tikzcd}
\end{equation}
We have $\mathrm{Im_H} ( E_+ ( \Sigma \partial ) ) \cong \mathrm{Im_H} ( E_+ ( \partial^\prime ) )$ since the suspension of the mapping cone sequence $C(f^\prime_2 ) \to C(g^\prime \circ f^\prime_2 ) \to C(g^\prime )$ is homotopy equivalent with the mapping cone sequence $C( l ) \to C( \tau \circ p_0 \circ l ) \to C( \tau \circ p_0 )$.
Especially $\mathrm{vol}^{-1} ( \mathrm{Im_H} ( E_+ ( \Sigma \partial ) ) ) = \mathrm{vol}^{-1} ( \mathrm{Im_H} ( E_+ ( \partial^\prime ) ) ) = \beta  ( [ \Lambda^\prime ] , [ \Lambda ] )$.
\end{remark}

In a simple application, we see that a properly extensible Brown functor possesses a natural HTQFT which is {\it not projective}.

\begin{Corollary}
\label{202101301508}
If a $\mathsf{Hopf}^\mathsf{bc,vol}_k$-valued  $d$-dimensional Brown functor $E$ is properly extensible, then the tensor product $\hat{\mathrm{PI}} (E) \otimes \check{\mathrm{PI}} (E)$ of cospanical and spanical path-integrals of $E$ is lifted to a $\mathsf{C}_k$-valued HTQFT with the supremal dimension $d$ given by $Z = \theta (E)^{-1} \cdot \left( \hat{\mathrm{PI}} (E) \otimes \check{\mathrm{PI}} (E) \right)$ (see Definition \ref{202001141545}).
In particular, $Z$ satisfies 
\begin{enumerate}
\item
$Z (K) = E(K) \otimes E(\Sigma K)$ for a pointed finite CW-space $K$ with $\dim K \leq (d-1)$.
\item
The induced homotopy invariant is given by $Z ( L ) = ( \dim_k E(L) ) \cdot 1 \in k^\ast$ for a pointed finite CW-space $L$ with $\dim L \leq d$.
\end{enumerate}
\end{Corollary}
\begin{proof}
By Theorem \ref{202001271630}, we have $\theta (E) \in \Theta ( \hat{\mathrm{PI}} (E) \otimes \check{\mathrm{PI}} (E) )$ (see Definition \ref{202003162027}).
Hence, the lift $\theta (E)^{-1} \cdot \left( \hat{\mathrm{PI}} (E) \otimes \check{\mathrm{PI}} (E) \right)$ is a symmetric monoidal functor.
Moreover the remaining claims follow from definitions and Corollary \ref{202006040804}.
\end{proof}

\begin{remark}
\label{202102161042}
There is a remark analogous to Remark \ref{202102161038}.
Let $K$ be a pointed finite CW-space with $\dim K \leq (d-1)$.
The mapping group representation on $Z(K)$ coincides with the representation induced by the Brown functor $E$ itself, in particular it is not projective.
In fact, we have $\theta ( [ K \stackrel{f}{\to} K \stackrel{Id_K}{\leftarrow} K ] ) = 1 \in k^\ast$ for a pointed map $f : K \to K$.
\end{remark}

\begin{Example}
One can apply Corollary \ref{202101301508} to the Brown functor $E$ in Example \ref{202102151212}.
In particular, $Z$ satisfies
\begin{enumerate}
\item
$Z (K) = k^{ [K, X] \times [\Sigma K , X ]}$ for a pointed finite CW-space $K$ with $\dim K \leq (d-1)$.
\item
The induced homotopy invariant is given by $Z ( L ) = | [ K , X ] | \cdot 1 \in k^\ast$ for a pointed finite CW-space $L$ with $\dim L \leq d$.
\end{enumerate}
\end{Example}

\begin{Example}
\label{202102160959}
Recall Example \ref{202102142024}.
Let $\widetilde{E}_\bullet$ be a $\mathsf{Hopf}^\mathsf{bc}_k$-valued homology theory and $q \in \Gamma ( \widetilde{E}_\bullet )$.
There exists a $\mathsf{C}_k$-valued homotopy-theoretic version of TQFT $Z$ with the supremal dimension $d = d ( \widetilde{E}_\bullet ; q)$ satisfying the following conditions.
It is obtained by considering $E= \widetilde{E}^\natural_q$ in Corollary \ref{202101301508}.
\begin{enumerate}
\item
$Z (K) = \widetilde{E}_q(K) \otimes \widetilde{E}_{q-1}(K)$ for a pointed finite CW-space $K$ with $\dim K \leq (d-1)$.
\item
The induced homotopy invariant is given by $Z ( L ) = ( \dim_k \widetilde{E}_q (L) ) \cdot 1 \in k^\ast$ for a pointed finite CW-space $L$ with $\dim L \leq d$.
\end{enumerate}
\end{Example}

\begin{Example}
\label{202102161025}
In the last part of subsection \ref{202102141031}, we give some examples of $\mathsf{Hopf}^\mathsf{bc}_k$-valued homology theories.
It is possible to construct concrete examples of Example \ref{202102160959} from such homology theories.
For convenience, we give one example here.
Consider the $\mathsf{Hopf}^\mathsf{bc}_k$-valued homology theory $\widetilde{E}_\bullet$ induced by the sphere spectrum in Example \ref{202102111049}.
For an integer $q > 0$ which clearly lies in $\Gamma ( \widetilde{E}_\bullet )$, we obtain a $\mathsf{Hopf}^\mathsf{bc,vol}_k$-valued Brown functor $\widetilde{E}^\natural_q$ with dimension $d = d( \widetilde{E}_\bullet ; q) = (q-1)$.
By a simple application of Example \ref{202102160959}, we obtain a HTQFT $Z$ with the supremal dimension $(q-1)$ such that 
\begin{enumerate}
\item
$Z (K) = k \left( \widetilde{\pi}^s_q(K) \otimes \widetilde{\pi}^s_{q-1} (K) \right)$ for a pointed finite CW-space $K$ with $\dim K \leq (q-2)$.
\item
The induced homotopy invariant is given by $Z ( L ) = | \widetilde{\pi}^s_q (L) | \in k^\ast$ for a pointed finite CW-space $L$ with $\dim L \leq (q-1)$.
\end{enumerate}
Moreover, for any $n \leq (q-1)$ there exists a $\mathsf{C}_k$-valued $n$-TQFT $Z$ satisfying the analogous conditions (see Proposition \ref{202102131241}).
\end{Example}

Before we close this subsection, we give some formulae obtained from Theorem \ref{202001271630}.
The formulae in Corollary \ref{202001271629} (Corollary \ref{202102151112}, resp.) show how the obstruction cocycles of some cospanical (spanical, resp.) path-integrals are related with each other.

\begin{Corollary}[Inversion formula 2]
\label{202001271629}
If a $\mathsf{Hopf}^\mathsf{bc,vol}_k$-valued  $d$-dimensional Brown functor $E$ is properly extensible, then the coboundary equation associated with $\hat{\mathrm{PI}} ( i^\ast_{d-1}  E) \otimes \hat{\mathrm{PI}} ( \Sigma^\ast_{d-1} E )$ is solvable.
Furthermore, the 1-cochain $\theta ( i^\ast_{d-1} E)$ gives a canonical solution :
\begin{align}
\hat{\omega} ( i^\ast_{d-1}  E)  \cdot  \hat{\omega} ( \Sigma^\ast_{d-1} E )  = \delta ( \theta ( i^\ast_{d-1} E) )  .
\end{align}
Especially, we have $\hat{\mathds{O}} ( i^\ast_{d-1} E  )  = \hat{\mathds{O}} ( \Sigma^\ast_{d-1} E  )^{-1}$.
\end{Corollary}
\begin{proof}
It follows from the following computations.
\begin{align}
\hat{\omega} ( i^\ast_{d-1}  E) 
&=
j^\ast_{d-1} (\hat{\omega} (E)) , \\
&=
j^\ast_{d-1} ( \check{\omega} (E) )^{-1} \cdot j^\ast_{d-1} ( \delta ( \theta (E))) ~~~(\because \mathrm{Theorem}~ \ref{202001271630} ), \\
&=
\hat{\omega} ( \Sigma^\ast_{d-1} E )^{-1} \cdot j^\ast_{d-1} ( \delta ( \theta (E))) ~~~(\because \mathrm{Corollary}~ \ref{202006041143} ) , \\
&=
\hat{\omega} ( \Sigma^\ast_{d-1} E)^{-1} \cdot \delta ( \theta ( i^\ast_{d-1} E) ) .
\end{align}
\end{proof}

\begin{Corollary}[Inversion formula 3]
\label{202102151112}
If a $\mathsf{Hopf}^\mathsf{bc,vol}_k$-valued  $d$-dimensional Brown functor $E$ is properly extensible, then the coboundary equation associated with $\check{\mathrm{PI}} ( i^\ast_{d-1}  E) \otimes \check{\mathrm{PI}} ( \Sigma^\ast_{d-1} E )$ is solvable.
Furthermore, the 1-cochain $\theta ( \Sigma^\ast_{d-1} E)$ gives a canonical solution :
\begin{align}
\check{\omega} ( i^\ast_{d-1}  E)  \cdot  \check{\omega} ( \Sigma^\ast_{d-1} E ) = \delta ( \theta ( \Sigma^\ast_{d-1} E) )  .
\end{align}
Especially, we have $\check{\mathds{O}} ( i^\ast_{d-1} E  )  = \check{\mathds{O}} ( \Sigma^\ast_{d-1} E  )^{-1}$.
\end{Corollary}
\begin{proof}
\begin{align}
&\check{\omega} ( i^\ast_{d-1}  E ) \cdot \check{\omega} ( \Sigma^\ast_{d-1} E ) , \\
=&
\left( \hat{\omega} ( i^\ast_{d-1} E )^{-1} \cdot \delta ( \theta ( i^\ast_{d-1} E) ) \right) \cdot \left(  \hat{\omega} ( \Sigma^\ast_{d-1} E)^{-1} \cdot \delta ( \theta ( \Sigma^\ast_{d-1} E) ) \right) ~~~(\because \mathrm{Theorem}~ \ref{202001271630} ) , \\
=&
\left( \hat{\omega} ( i^\ast_{d-1} E ) \cdot \hat{\omega} ( \Sigma^\ast_{d-1} E ) \right)^{-1} \cdot \delta ( \theta ( i^\ast_{d-1} E) \cdot \theta ( \Sigma^\ast_{d-1} E) ) , \\
=&
\delta ( \theta ( i^\ast_{d-1} E) )^{-1} \cdot \delta ( \theta ( i^\ast_{d-1} E) \cdot \theta ( \Sigma^\ast_{d-1} E) ) ~~~ ( \because \mathrm{Corollary}~\ref{202001271629} ) , \\
=&
\delta ( \theta ( \Sigma^\ast_{d-1} E) ) .
\end{align}
\end{proof}

\begin{Corollary}
Let $\widetilde{E}_\bullet$ be a $\mathsf{Hopf}^\mathsf{bc}_k$-valued homology theory.
If $q, (q+1) \in \Gamma ( \widetilde{E}_\bullet )$ with $d= d(\widetilde{E}_\bullet ; q)$, then we have 
\begin{align}
j^\ast_d ( \hat{\mathds{O}} ( \widetilde{E}^\natural_{q+1} ) )  &= \hat{\mathds{O}} ( \widetilde{E}^\natural_q )^{-1} , \\
j^\ast_d ( \check{\mathds{O}} ( \widetilde{E}^\natural_{q+1} ) )  &= \check{\mathds{O}} ( \widetilde{E}^\natural_q )^{-1} .
\end{align}
\end{Corollary}
\begin{proof}
It is immediate from the suspension isomorphism $\widetilde{E}^\natural_q \cong \Sigma^\ast_d \widetilde{E}^\natural_{q+1}$ and Corollary \ref{202001271629}, \ref{202102151112}.
\end{proof}

\subsection{Dimension reduction}
\label{202002211404}

In this subsection, we give the first application of the results in subsection \ref{202102150945}.
We prove that {\it the dimension reduction} of the path-integral of a Brown functor is lifted to a homotopy-theoretic version of TQFT.
The dimension reduction in the topological field theory is a technique to derive a field theory by taking a product of manifolds with a circle $\mathbb{T}$.
We generalize the technique to deal with pointed finite CW-spaces by considering $\mathbb{T}^+ = \mathbb{T} \amalg \{ \mathrm{pt} \}$ instead of $\mathbb{T}$ as follows.

\begin{Defn}
\label{202102141850}
Recall Definition \ref{202102231038}.
Let $X$ be a pointed finite CW-space with $\dim X = r$.
For a HTQFT $Z$ with the supremal dimension $d$ we denote by the induced HTQFT $W^\ast_X Z = Z \circ W_X$ with the supremal dimension $(d-r)$.
{\it The dimension reduction of $Z$} is the induced HTQFT $W^\ast_{\mathbb{T}^+} Z$ with the supremal dimension $(d-1)$.
\end{Defn}

\begin{remark}
We have $W^\ast_{\mathbb{T}^+} Z ( K ) \cong Z ( K \wedge \mathbb{T}^+ )$.
For a CW-space $L$ without a basepoint, let $Z^\prime ( L ) = Z (L^+ )$.
Then we have $W^\ast_{\mathbb{T}^+} Z^\prime ( L ) \cong Z^\prime ( L \times \mathbb{T} )$ since $(L \times \mathbb{T} )^+ \cong L^+ \wedge \mathbb{T}^+$.
Based on this observation, we use the terminology {\it the dimension reduction} (for example see \cite{freed2009remarks}).
\end{remark}

\begin{Lemma}
\label{202002181018}
Let $E$ be a $\mathsf{Hopf}^\mathsf{bc,vol}_k$-valued $d$-dimensional Brown functor and $W^\ast_{\mathbb{T}^+} E$ be the induced $(d-1)$-dimensional Brown functor.
We have a natural isomorphism, $W^\ast_{\mathbb{T}^+} E \cong i^\ast_{d-1} E \otimes \Sigma^\ast_{d-1} E$.
\end{Lemma}
\begin{proof}
Let $S^n$ be the pointed $n$-sphere.
Consider the mapping cone sequence $S^0 \to \mathbb{T}^+ \to S^1 \stackrel{\partial}{\to} \Sigma S^0$.
Let $K$ be a pointed finite CW-space with $\dim K \leq (d-1)$.
By Proposition \ref{202101271327}, we obtain an exact sequence for a pointed finite CW-space $K$ with $\dim \leq (d-1)$ :
\begin{align}
E(K \wedge S^0) \stackrel{\xi}{\to} E(K \wedge \mathbb{T}^+ ) \to E ( K \wedge S^1 ) \stackrel{\xi^\prime}{\to} E(K \wedge \Sigma S^0) .
\end{align}
The collapsing map $\mathbb{T}^+ \to S^0$ induces a retract of $\xi$ in the category $\mathsf{Hopf}^\mathsf{bc,vol}_k$.
On the other hand, $\xi^\prime$ is trivial.
In fact, the pointed map $\partial : S^1 \to \Sigma S^0 \cong S^1$ is null-homotopic.
Thus, we obtain a natural isomorphism $E(K \wedge \mathbb{T}^+ ) \cong E(K\wedge S^0) \otimes E(K \wedge S^1) \cong E(K) \otimes E(\Sigma K)$.
It completes the proof.
\end{proof}

\begin{theorem}
\label{202006051120}
If a $\mathsf{Hopf}^\mathsf{bc,vol}_k$-valued  $d$-dimensional Brown functor $E$ is properly extensible, then the coboundary equations associated with the dimension reductions $W^\ast_{\mathbb{T}^+} \hat{\mathrm{PI}} ( E )$ and $W^\ast_{\mathbb{T}^+} \check{\mathrm{PI}} ( E )$ are solvable respectively.
Furthermore, we have following canonical solutions.
\begin{enumerate}
\item
$\omega ( W^\ast_{\mathbb{T}^+}  \hat{\mathrm{PI}} ( E ) ) = \delta ( \theta ( i^\ast_{d-1} E ) )$.
\item
$\omega ( W^\ast_{\mathbb{T}^+} \check{\mathrm{PI}} ( E ) ) = \delta ( \theta ( \Sigma^\ast_{d-1} E ) )$.
\end{enumerate}
In particular, the obstruction classes $W^\ast_{\mathbb{T}^+} \hat{\mathds{O}} ( E ), W^\ast_{\mathbb{T}^+} \check{\mathds{O}} ( E )$ vanish.
\end{theorem}
\begin{proof}
We have following natural isomorphisms of symmetric monoidal projective functors in the strong sense :
\begin{align}
W^\ast_{\mathbb{T}^+} \hat{\mathrm{PI}} ( E )
&\cong
\hat{\mathrm{PI}} ( W^\ast_{\mathbb{T}^+} E) , \\
&\cong
\hat{\mathrm{PI}} ( i^\ast_{d-1} E) \otimes \hat{\mathrm{PI}} ( \Sigma^\ast_{d-1} E ) ~~~(\because \mathrm{Lemma}~ \ref{202002181018} ) , \\
&\cong
\hat{\mathrm{PI}} ( i^\ast_{d-1} E) \otimes \check{\mathrm{PI}} ( i^\ast_{d-1} E ) ~~~(\because \mathrm{Theorem}~  \ref{202001141422} ) .
\end{align}
Then the first claim follows from Theorem \ref{202001271630}.
We leave the second claim to the readers.
\end{proof}

\begin{Corollary}
\label{202102161020}
If a $\mathsf{Hopf}^\mathsf{bc,vol}_k$-valued  $d$-dimensional Brown functor $E$ is properly extensible, then the dimension reductions $W^\ast_{\mathbb{T}^+} \hat{\mathrm{PI}} ( E )$ and $W^\ast_{\mathbb{T}^+} \check{\mathrm{PI}} ( E )$ are lifted to HTQFT's.
\end{Corollary}
\begin{proof}
We have canonical lifts $\theta ( i^\ast_{d-1}E )^{-1} \cdot W^\ast_{\mathbb{T}^+} \hat{\mathrm{PI}} ( E )$ and $\theta ( \Sigma^\ast_{d-1}E )^{-1} \cdot W^\ast_{\mathbb{T}^+} \check{\mathrm{PI}} ( E )$ by Theorem \ref{202006051120}.
\end{proof}

\begin{remark}
\label{202102161022}
By the isomorphism $W^\ast_{\mathbb{T}^+} \hat{\mathrm{PI}} (E) \cong \hat{\mathrm{PI}} ( i^\ast_{d-1} E) \otimes \check{\mathrm{PI}} ( i^\ast_{d-1} E )$ in the proof of Theorem \ref{202006051120}, the HTQFT's in Corollary \ref{202102161020} are restrictions of HTQFT's in Corollary \ref{202101301508}.
\end{remark}

\begin{Example}
The above corollaries could be applied to the Brown functors given by Example \ref{202102151212}.
In other words, a nice homotopy commutative H-group provides $\mathsf{C}_k$-valued HTQFT's through the dimension reduction.
\end{Example}

We have another class of examples related with Example \ref{202102142024}.
Before we give it, we prove the following lemma.

\begin{Lemma}
\label{202002181042}
Let $\widetilde{F}_\bullet$ be a $\mathsf{Hopf}^\mathsf{bc}_k$-valued reduced homology theory.
Suppose that $\widetilde{F}_\bullet$ is a dimension reduction of another homology theory, i.e. $\widetilde{F}_\bullet = W^\ast_{\mathbb{T}^+} \widetilde{E}_\bullet$ for some $\mathsf{Hopf}^\mathsf{bc}_k$-valued homology theory $\widetilde{E}_\bullet$.
We have 
\begin{align}
\Gamma ( \widetilde{F}_\bullet ) = \Gamma (\widetilde{E}_\bullet ) \cap \left( \Gamma (\widetilde{E}_\bullet ) + 1\right) .
\end{align}
\end{Lemma}
\begin{proof}
Note that the biproduct in the abelian category $\mathcal{A} = \mathsf{Hopf}^\mathsf{bc}_k$ is the tensor product of Hopf algebras.
By Lemma \ref{202002181018}, it suffices to prove that $\Gamma (\widetilde{E}_\bullet \otimes \widetilde{E}_{\bullet -1} ) = \Gamma (\widetilde{E}_\bullet ) \cap \Gamma (\widetilde{E}_{\bullet -1})$.
$\Gamma (\widetilde{E}_\bullet ) \cap \Gamma (\widetilde{E}_{\bullet -1})\subset \Gamma (\widetilde{E}_\bullet \otimes \widetilde{E}_{\bullet -1} )$ is clear.
Let $q \in \Gamma (\widetilde{E}_\bullet \otimes \widetilde{E}_{\bullet -1} )$, i.e. the Hopf algebra $\widetilde{E}_q ( S^0 ) \otimes \widetilde{E}_{q-1} ( S^0)$ has a finite volume.
We claim that $\widetilde{E}_q ( S^0 )$ and $\widetilde{E}_{q-1} ( S^0 )$ have a finite volume.
More generally, for a bicommutative Hopf algebras $A,B$, if the tensor product $A \otimes B$ has a finite volume, then $A$ has a finite volume.
In fact, the composition of the inclusion $i : A \to A \otimes B$ and the normalized cointegral on $A\otimes B$ induces a normalized cointegral on $A$.
In the same manner, $A$ is proved to have a normalized integral by using the projection $A\otimes B \to A$.
Likewise, the Hopf algebra $B$ has a normalized integral and a normalized cointegral.
In particular, the inverse volume of $A,B$ are defined.
By Theorem \ref{202002211501}, we obtain $\mathrm{vol}^{-1} (A) \cdot \mathrm{vol}^{-1} (B) = \mathrm{vol}^{-1} (A \otimes B)$.
The inverse volume $\mathrm{vol}^{-1} (A \otimes B)$ is invertible so that $\mathrm{vol}^{-1} (A)$ is invertible.
It proves that $A$ has a finite volume.
\end{proof}

\begin{Example}
\label{202102160937}
Let $\widetilde{F}_\bullet$ be a $\mathsf{Hopf}^\mathsf{bc}_k$-valued reduced homology theory.
Suppose that $\widetilde{F}_\bullet$ is a dimension reduction of another homology theory $\widetilde{E}_\bullet$ as before.
For $q \in \Gamma ( \widetilde{F}_\bullet ) = \Gamma (\widetilde{E}_\bullet ) \cap \left( \Gamma (\widetilde{E}_\bullet ) + 1\right)$, the cospanical path-integral $\hat{\mathrm{PI}} ( \widetilde{F}^\natural_q )$ is lifted to a HTQFT.
In fact, we have $\widetilde{F}^\natural_q  \cong W^\ast_{\mathbb{T}^+} \widetilde{E}^\natural_q$ by definitions so that $\hat{\mathrm{PI}} ( \widetilde{F}^\natural_q ) \cong W^\ast_{\mathbb{T}^+} \hat{\mathrm{PI}} ( \widetilde{E}^\natural_q )$.
Especially we have $\omega ( \hat{\mathrm{PI}} ( \widetilde{F}^\natural_q ) ) = \delta ( \theta ( i^\ast_{d-1} \widetilde{E}^\natural_q ) )$ by Theorem \ref{202006051120} where $d = d ( \widetilde{E}_\bullet ; q )$.
There is similar application for the spanical path-integral $\check{\mathrm{PI}} ( \widetilde{F}^\natural_q )$.
\end{Example}

\begin{Example}
There are various examples of $\mathsf{Hopf}^\mathsf{bc}_k$-valued homology theories.
See the last part of subsection \ref{202102141031}.
One could construct concrete examples of Example \ref{202102160937} from such homology theories.
For convenience, we give one example here.
Let $\widetilde{E}_\bullet$ be the $\mathsf{Hopf}^\mathsf{bc}_k$-valued homology theory induced by the stable homotopy theory following Example \ref{202102111049}, and $\widetilde{F}_\bullet$ be the dimension reduction of $\widetilde{E}_\bullet$.
Then we obtain $\Gamma ( \widetilde{F}_\bullet ) = \left( \mathbb{Z} \backslash \{ 0 \} \right) \cap \left( \mathbb{Z} \backslash \{ 1 \} \right) = \mathbb{Z} \backslash \{ 0 , 1\}$ by Lemma \ref{202002181042}.
For an integer $q$ with $q \geq 2$, the cospanical path-integral $\hat{\mathrm{PI}} ( \widetilde{F}^\natural_q )$ is lifted to a $\mathsf{C}_k$-valued HTQFT with the supremal dimension $d ( \widetilde{F}_\bullet ;q ) = (q-2)$.
Note that the HTQFT is a restriction of that in Example \ref{202102161025} (see Remark \ref{202102161022}).
\end{Example}


\subsection{Bounded-below homology theory}
\label{202001210015}

In this subsection, we study the obstruction cocycles associated with bounded homology theories.
In particular, we verify that some obstruction classes vanish mainly by using the inversion formulae.
The result implies a generalization of DW and TV TQFT's.

\begin{remark}
One could modify the results of this subsection for application to bounded-above homology theories.
\end{remark}

Let $\widetilde{E}_\bullet$ be a $\mathsf{Hopf}^\mathsf{bc}_k$-valued reduced homology theory which is bounded below.
In other words, there exists $q_0 \in \mathbb{Z}$ such that $q < q_0$ implies $\widetilde{E}_q ( K ) \cong k$ for any pointed finite CW-space $K$.
In this subsection, we prove that the obstruction class $\hat{\mathds{O}}(\widetilde{E}^\natural_q)$ vanishes for $q \in \Gamma ( \widetilde{E}_\bullet )$ such that $m( \widetilde{E}_\bullet ; q) = -\infty$ (see Definition \ref{202001202100}).

\begin{Example}
Let $\widetilde{D}_\bullet$ be a generalized homology theory which is bounded below.
Then the induced $\mathsf{Hopf}^\mathsf{bc}_k$-valued homology theory $\widetilde{E}_\bullet = k \widetilde{D}_\bullet$ is bounded below.
Let $q \in \mathbb{Z}$ such that $q^\prime < q$ implies that the $q^\prime$-th coefficient $\widetilde{D}_{q^\prime} ( S^0 )$ is finite and its order is coprime to the characteristic of $k$.
Then $q \in \Gamma ( \widetilde{E}_\bullet)$ satisfies $m ( \widetilde{E}_\bullet ; q ) = -\infty$ by definitions.
For example, one could consider $\widetilde{D}_\bullet = \widetilde{\mathrm{MO}}_\bullet$ with an arbitrary $q \in \Gamma ( \widetilde{E}_\bullet) = \mathbb{Z}$ in Example \ref{202102111049}.
\end{Example}

\begin{Example}
Let $A$ be a bicommutative Hopf algebra over $k$.
Let $\widetilde{E}_\bullet (-) = \widetilde{H}_\bullet ( - ; A)$ be the reduced ordinary homology theory with coefficients in $A$.
The homology theory $\widetilde{E}_\bullet$ is bounded below.
Suppose that the Hopf algebra $A$ has a finite volume.
Then we have $\Gamma ( \widetilde{E}_\bullet ) = \mathbb{Z}$ and $m ( \widetilde{E}_\bullet ; q ) = -\infty$ for any $q \in \Gamma ( \widetilde{E}_\bullet )$ by definitions.
The application of Corollary \ref{202006080930} to such $\widetilde{E}_\bullet$ gives a generalization of abelian Dijkgraaf-Witten TQFT and bicommutative Turaev-Viro TQFT to arbitrary field $k$, $q$ and top dimension $n$ of manifolds.
DW TQFT is based on $k= \mathbb{C}$, $q=1$, arbitrary $n$, and TV TQFT is based on an algebraically closed field $k$ with characteristic zero, $q=1$, $n=3$.
See the following subsubsections.
\end{Example}

\begin{Example}
Recall Example \ref{202102111055}.
If $\widetilde{D}_\bullet$ is bounded below, then so is $\widetilde{E}_\bullet$.
By definitions, we have $\Gamma ( \widetilde{E}_\bullet ) = \mathbb{Z}$ and $m ( \widetilde{E}_\bullet ; q ) = -\infty$ for any $q \in \Gamma ( \widetilde{E}_\bullet )$.
\end{Example}

\begin{Defn}
\label{202001252241}
We define a normalized 1-cochain $\theta_{\leq q} (\widetilde{E}_\bullet )$ of the symmetric monoidal category $\mathsf{Cosp}^\simeq_{\leq \infty} ( \mathsf{CW}^\mathsf{fin}_\ast )$ with coefficients in the multiplicative group $k^\ast$ by
\begin{align}
\theta_{\leq q} (\widetilde{E}_\bullet )
\stackrel{\mathrm{def.}}{=}
\prod_{l \geq 0} \theta ( \widetilde{E}^\natural_{q-l} )^{(-1)^l} 
\end{align}
\end{Defn}

\begin{theorem}
\label{202001141430}
Let $\widetilde{E}_\bullet$ be a $\mathsf{Hopf}^\mathsf{bc}_k$-valued reduced homology theory which is bounded below.
For $q \in \Gamma ( \widetilde{E}_\bullet )$ such that $m ( \widetilde{E}_\bullet ; q ) = -\infty$, the coboundary equation associated with $\hat{\mathrm{PI}} ( \widetilde{E}^\natural_q )$ is solvable.
Furthermore, we have a canonical solution as follows.
\begin{align}
\hat{\omega} ( \widetilde{E}^\natural_q ) = \delta ( \theta_{\leq q} (\widetilde{E}_\bullet ) ) .
\end{align}
In particular, the obstruction class $\hat{\mathds{O}} ( \widetilde{E}^\natural_q )$ vanishes.
\end{theorem}
\begin{proof}
Note that $m ( \widetilde{E}_\bullet ; q ) = -\infty$ is equivalent with that $q^\prime \leq q$ implies $q^\prime \in \Gamma ( \widetilde{E}_\bullet )$.
By Corollary \ref{202001271629}, we have $\omega ( \hat{\mathrm{PI}} ( \widetilde{E}^\natural_q ) ) = \omega ( \hat{\mathrm{PI}} ( \widetilde{E}^\natural_{q-1} ) )^{-1} \cdot \delta (\theta (\widetilde{E}^\natural_q ) )$ due to $d = \infty$.
We repeat this formula until the integer $q_0$.
Since the homology theory $\widetilde{E}_\bullet$ is assumed to be bounded below, we obtain the claim.
\end{proof}

\begin{Corollary}
\label{202102111300}
For $q \in \Gamma ( \widetilde{E}_\bullet )$ such that $m ( \widetilde{E}_\bullet ; q ) = -\infty$, the coboundary equation associated with $\check{\mathrm{PI}} ( \widetilde{E}^\natural_q )$ is solvable.
Furthermore, we have a canonical solution as follows.
\begin{align}
\check{\omega} ( \widetilde{E}^\natural_q ) = \delta ( \theta_{\leq (q-1)} (\widetilde{E}_\bullet ) ) .
\end{align}
In particular, the obstruction class $\check{\mathds{O}} ( \widetilde{E}^\natural_q )$ vanishes.
\end{Corollary}
\begin{proof}
The Theorem \ref{202001271630} implies $\omega ( \check{\mathrm{PI}} ( \widetilde{E}^\natural_q ) ) = \omega ( \hat{\mathrm{PI}} ( \widetilde{E}^\natural_q ) )^{-1} \cdot \delta (\theta (\widetilde{E}^\natural_q ) )$.
By Theorem \ref{202001141430}, we have $\omega ( \check{\mathrm{PI}} ( \widetilde{E}^\natural_q ) ) = \delta ( \theta_{\leq q} (\widetilde{E}_\bullet )^{-1} ) \cdot \delta (\theta (\widetilde{E}^\natural_q ) )  =  \delta ( \theta_{\leq (q-1)} (\widetilde{E}_\bullet ) )$.
\end{proof}

\begin{Corollary}
\label{202006080930}
Let $q \in \Gamma ( \widetilde{E}_\bullet )$ such that $m ( \widetilde{E}_\bullet ; q ) = -\infty$.
Then there exists a $\mathsf{C}_k$-valued homotopy-theoretic version of TQFT $Z$ with the supremal dimension $\infty$ which satisfies the following conditions :
\begin{enumerate}
\item 
The diagram below commutes strictly.
\begin{equation}
\begin{tikzcd}
\mathsf{Ho} ( \mathsf{CW}^\mathsf{fin}_{\ast} ) \ar[d, hookrightarrow] \ar[r, " \widetilde{E}^\natural_q"] & \mathsf{Hopf}^\mathsf{bc,vol}_k \ar[d, hookrightarrow] \\
\mathsf{Cosp}^\simeq_{\leq \infty} ( \mathsf{CW}^\mathsf{fin}_\ast ) \ar[r, "Z"] & \mathsf{C}_k
\end{tikzcd}
\end{equation}
\item
The induced homotopy invariant is given by
\begin{align}
Z ( L ) = \left( \prod_{l \geq 0} \dim_k ( \widetilde{E}_{q-l} ( L))^{(-1)^l} \right) \cdot 1 \in k^\ast .
\end{align}
\end{enumerate}
\end{Corollary}
\begin{proof}
By Theorem \ref{202001141430}, $\theta_{\leq q} (\widetilde{E}_\bullet) \in \Theta ( \hat{\mathrm{PI}} ( \widetilde{E}^\natural_q ) )$.
Then the lift $Z = \theta_{\leq q} (\widetilde{E}_\bullet)^{-1} \cdot \hat{\mathrm{PI}} (\widetilde{E}^\natural_q)$ satisfies all the conditions by Example \ref{202102121407} and Lemma \ref{202006041011}.
Note that $d( \widetilde{E}_\bullet ; q) = q - m( \widetilde{E}_\bullet ; q) = \infty$.
\end{proof}

\begin{remark}
There is a remark analogous to Remark \ref{202102161038}, \ref{202102161042}.
Let $K$ be a pointed finite CW-space.
By Corollary \ref{202006080930}, the induced mapping group representation on $Z(K)$ coincides with the representation induced by the homology theory $\widetilde{E}_q$ itself.
\end{remark}

\begin{Example}
All of the homology theories in the beginning of this subsection extends to a $\mathsf{C}_k$-valued HTQFT with the supremal dimension $\infty$ as Corollary \ref{202006080930}.
By Proposition \ref{202102131241}, such a HTQFT induces a $\mathsf{C}_k$-valued $n$-TQFT for any $n \in \mathbb{N}$.
\end{Example}

\subsubsection{The untwisted Dijkgraaf-Witten-Freed-Quinn theory}

In this subsubsection, we give an overview of the Dijkgraaf-Witten-Freed-Quinn theory.
The Dijkgraaf-Witten TQFT \cite{DW} \cite{Wakui} is a 3-dimensional unitary TQFT constructed by using a triangulation of 3-manifolds and a 3-cocycle of a finite group $G$.
The construction is streamlined and generalized to any dimension and manifolds which might not be triangulable \cite{FQ} \cite{SV}.
The DW TQFT associated with the trivial group cocycle is called an untwisted DW TQFT.
For $n \in \mathbb{N}$, the untwisted DW invariant is a homotopy invariant of closed $n$-manifolds $M$ given by $Z^1_{\mathrm{DW},G} (M) = \prod_{i} \frac{|\mathrm{Hom} ( \pi_1 ( M_i ) , G)|}{|G|}$ where $i$ runs over the components of $M$.
The untwisted DW invariant extends to a unitary $n$-dimensional TQFT $Z^1_{\mathrm{DW},G}$ valued at the category of Hilbert spaces.
For simplicity, let $G$ be an abelian group.
It assigns a Hilbert space $Z^1_{\mathrm{DW},G} (N) \cong \mathbb{C}^{H^1( N ; G)}$ to a closed $(n-1)$-manifold $N$.
Its inner product is given by $\langle f, g \rangle =  |H^0 ( N ; G)|^{-1} \sum_{x} \overline{f(x)} g(x)$.
For an $n$-manifold $M$ possibly with boundaries, let $\mathrm{W} = ( \emptyset \to M \leftarrow \partial M)$ be the induced cobordism from the empty manifold to the boundary $\partial M$.
$Z^1_{\mathrm{DW},G} (\mathrm{W}) \in Z^1_{\mathrm{DW},G} ( \partial M )$ assigns the groupoid cardinality \cite{baez2009higher} of the homotopy fiber groupoid $r^{-1} (Q)$ to $[Q] \in H^1( \partial M ; G)$.
Here the functor $r : \mathcal{B} ( M ; G ) \to \mathcal{B} ( \partial M ; G)$ is the restriction where $\mathcal{B} ( X ; G )$ is the finite groupoid of principal $G$-bundles.
We identify $H^1( X ; G)$ with the isomorphism classes of principal $G$-bundles over $X$.
For an $n$-manifold $M$ possibly with boundaries, we have
\begin{align}
Z^1_{\mathrm{DW},G} (\mathrm{W}) : [Q] \mapsto 
\begin{cases}
\frac{|H^1 ( M , \partial M ; G) |}{|H^0 ( M , \partial M ; G)|} & (\exists P \in \mathcal{B} ( M ; G) \mathrm{~s.t.~} P|_{\partial M} \cong Q ) , \\
0 & (\mathrm{otherwise}) .
\end{cases}
\end{align}

\subsubsection{Reproduction of the untwisted DW TQFT}
\label{202002041049}
We reproduce the untwisted DW TQFT in \cite{FQ}.
Let $G$ be a finite abelian group and $\widetilde{H}^\bullet ( - ; G)$ be the reduced ordinary cohomology theory with coefficients in $G$.
It induces a $\mathsf{Hopf}^\mathsf{bc}_k$-valued homology theory by $\widetilde{E}_\bullet = \mathbb{C}^{\widetilde{H}^\bullet ( - ; G)}$.
For such $\widetilde{E}_\bullet$ consider $Z$ in Corollary \ref{202006080930} for $q=1$.
In other words, $Z$ is the $\mathsf{C}_k$-valued homotopy-theoretic version of TQFT with the supremal dimension $\infty$ :
\begin{align}
Z = \theta_{\leq 1} ( \widetilde{E}_\bullet )^{-1} \cdot  \hat{\mathrm{PI}} ( \widetilde{E}^\natural_1 )
=
\frac{\theta (\widetilde{E}^\natural_0)}{\theta (\widetilde{E}^\natural_1)} \cdot  \hat{\mathrm{PI}} ( \widetilde{E}^\natural_1 ) .
\end{align}
For a pointed finite CW-space $K$, the corresponding object is $Z(K) = \widetilde{E}_1 ( K) = \mathbb{C}^{\widetilde{H}^1 (K ; G)}$.
It has a natural Hilbert space structure as follows.
On the one hand, the space $K$ has a duality in the cospan category $\mathsf{Cosp}^\simeq_{\leq \infty} ( \mathsf{CW}^\mathsf{fin}_\ast )$ which is given by $\mathrm{coev}_K = [\ast \to K \stackrel{\nabla}{\leftarrow} K \vee K]$ and $\mathrm{ev}_K = [K \vee K \stackrel{\nabla}{\to} K \leftarrow \ast ]$.
By the conjugate on $Z(K) = \mathbb{C}^{\widetilde{H}^1 (K; G)}$, we obtain a Hilbert space structure on $Z(K)$ :
\begin{align}
\langle f , g \rangle^\prime = ( Z (\mathrm{ev}_K) ) ( \bar{f} \otimes g ) .
\end{align}
We regard $Z(K)$ as a Hilbert space from now on.

\begin{Lemma}
\label{202102122037}
We have
\begin{align}
\langle f , g \rangle^\prime = \langle f , g \rangle .
\end{align}
\end{Lemma}
\begin{proof}
It suffices to prove that $( Z (\mathrm{ev}_K) ) ( f \otimes g ) = |H^0 ( K ; G)|^{-1} \sum_x f(x) g(x)$.
The morphism $(\hat{\mathrm{PI}} (\widetilde{E}^\natural_1) ) (\mathrm{ev}_K)$ coincides with $\int_\Lambda$ by definitions where $\Lambda = ( A \otimes A \stackrel{\nabla}{\to} A \leftarrow k)$ and $A = Z(K) = \mathbb{C}^{\widetilde{H}^1 (K; G)}$.
The normalized integral along the unit $k \to A$ is the normalized cointegral of $A$.
Note that the normalized cointegral of $A=\mathbb{C}^{\widetilde{H}^1 (K; G)}$ is given by $\sigma^A (f) = |\widetilde{H}^1 (K; G)|^{-1} \sum_x f(x)$.
Hence, $((\hat{\mathrm{PI}} (\widetilde{E}^\natural_1) ) (\mathrm{ev}_K)) (f \otimes g) = |H^1 ( K ; G)|^{-1} \sum_x f(x) g(x)$.
By the definition of $Z$, we obtain the result.
\end{proof}

\begin{remark}
In \cite{FQ}, the assigned linear homomorphisms and the inner products of the assigned vector spaces are related to some groupoid cardinalities \cite{baez2009higher}.
Such groupoid cardinalities naturally appear due to the canonical solution $\theta_{\leq 1} ( \widetilde{E}_\bullet )$ according to the proof of Lemma \ref{202102122037} and Theorem \ref{202102122133}.
\end{remark}

\begin{theorem}
\label{202102122133}
For a closed $(n-1)$-manifold $N$, we have a natural isomorphism between Hilbert spaces $Z ( \Phi_{\infty,n} (N) ) \cong Z^1_{\mathrm{DW},G} (N)$ where $\Phi_{\infty,n}$ is defined in Definition \ref{202006081111}.
Under the isomorphism, we have $Z ( \Phi_{\infty,n} (\mathrm{W} )) = Z^1_{\mathrm{DW},G} ( \mathrm{W} )$ for an $n$-cobordism $\mathrm{W}$.
\end{theorem}
\begin{proof}
We sketch the proof.
The underlying vector spaces of $Z ( \Phi_{\infty,n} (N) )$ and $Z^1_{\mathrm{DW},G} ( N )$ are the same.
Moreover their inner products coincide with each other by Lemma \ref{202102122037}.
Next we compute $Z(\Phi_{\infty, n} (\mathrm{W}))$ for an $n$-cobordism $\mathrm{W} = ( \emptyset \to M \leftarrow \partial M)$.
It induces a cospan (\ref{202002070939}) in $\mathsf{Hopf}^\mathsf{bc,vol}_\mathbb{C}$.
By Example \ref{202102112335}, the integral along this cospan coincides with $Z^1_{\mathrm{DW},G} ( \mathrm{W} )$.
Since the pairings on $Z(N$) induced by $\mathrm{ev}_N$ coincide with each other, $Z ( \Phi_{\infty,n} (\mathrm{W} )) = Z^1_{\mathrm{DW},G} ( \mathrm{W} )$ holds for any cobordism $\mathrm{W}$.
It completes the proof.
\begin{equation}
\label{202002070939}
\begin{tikzcd}
& \mathbb{C}^{H^1 ( M ; G)} & \\
\mathbb{C}^{H^1 ( \emptyset ; G)} \ar[ur] & & \mathbb{C}^{H^1 ( \partial M ; G)} \ar[ul]
\end{tikzcd}
\end{equation}
\end{proof}

\begin{remark}
By Theorem \ref{202102122133}, we regard $Z = \theta_{\leq 1} ( \widetilde{E}_\bullet )^{-1} \cdot  \hat{\mathrm{PI}} ( \widetilde{E}^\natural_1 )$ as an extension of the untwisted DW TQFT $Z^1_{\mathrm{DW},G}$ to $\mathsf{Cosp}^\simeq_{\leq \infty} ( \mathsf{CW}^\mathsf{fin}_\ast )$.
Note that the terminology {\it extension} is not the same as the {\it extended TQFT} in the literature which deals with codimension more than 1.
\end{remark}

In the same manner, the untwisted higher abelian Dijkgraaf-Witten (unextended) TQFT \cite{monnier2015higher} is reproduced by our result.
Let $Z^q_{\mathrm{DW},G}$ be the untwisted higher abelian Dijkgraaf-Witten $n$-TQFT obtained by $q$-form gauge fields over $G$, and $Z = \theta_{\leq q} ( \widetilde{E}_\bullet )^{-1} \cdot \hat{\mathrm{PI}} ( \widetilde{E}^\natural_q )$.
The proof is similar with that of the above proposition.

\begin{prop}
For a closed $(n-1)$-manifold $N$, we have a natural isomorphism between Hilbert spaces $Z ( \Phi_{\infty,n} (N) ) \cong Z^q_{\mathrm{DW},G} (N)$.
Under the isomorphism, we have $Z ( \Phi_{\infty,n} (\mathrm{W} )) = Z^q_{\mathrm{DW},G} ( \mathrm{W} )$ for an $n$-cobordism $\mathrm{W}$.
\end{prop}
\subsubsection{Reproduction of TV TQFT}

The Turaev-Viro-Barrett-Westbury theory \cite{TV}\cite{BW}\cite{kirillov2010turaev} provides a 3-dimensional TQFT $Z_\mathrm{TV,\mathcal{C}}$ starting from a spherical fusion category $\mathcal{C}$ over an algebraically closed field $k$ of characteristic zero.
A typical example of $\mathcal{C}$ is given by the (finite-dimensional) representation category $\mathcal{C}= \mathsf{Rep}(A)$ of a finite-dimensional involutory Hopf algebra over $A$ the field $k$.
Especially, a finite-dimensional bicommutative Hopf algebra $A$ over $k$ induces a TV TQFT $Z_{\mathrm{TV},\mathsf{Rep}(A)}$.

\begin{prop}
Let $\widetilde{E}_\bullet$ be the ordinary homology theory $\widetilde{H}_\bullet ( - ; A)$ with coefficients in $A$.
For $Z = \theta_{\leq 1} ( \widetilde{E}_\bullet )^{-1} \cdot \hat{\mathrm{PI}} ( \widetilde{E}^\natural_1 )$, the composition $U \circ Z \circ \Phi_{\infty, 3}$ is naturally isomorphic to the TV TQFT $Z_{\mathrm{TV},\mathsf{Rep} ( A )}$.
Here, $U : \mathsf{C}_k \to \mathsf{Vec}^\mathsf{fin}_k$ is the forgetful functor.
\end{prop}
\begin{proof}
By Corollary \ref{202006091350}, the Hopf algebra $A$ is bisemisimple.
Hence it is a function Hopf algebra $A= k^G$ induced by a finite abelian group $G$ since $k$ is algebraically closed.
The representation category $\mathsf{Rep} ( k^G )$ coincides with the category $\mathsf{Vec}^G$ of $G$-graded (finite-dimensional) vector spaces.
The TV TQFT $Z_{\mathrm{TV},\mathsf{Vec}^G}$ is isomorphic to the untwisted DW TQFT $Z^1_{\mathrm{DW},G}$ where $Z^1_{\mathrm{DW},G}$ is constructed over the field $k$ not the complex field $\mathbb{C}$.
Therefore, we obtain the claim through an analogous result of the preceding subsubsection for the field $k$.
\end{proof}
\appendix
\section{Symmetric monoidal projective functor}
\label{202002070917}

In this appendix, we give an overview of symmetric monoidal projective functors and their associated obstruction classes.
Fix a symmetric monoidal category $\mathcal{E}$ satisfying the following assumptions : 
\begin{enumerate}
\item
The endomorphism set of the unit object $End_\mathcal{E} ( \mathds{1} )$ consists of automorphisms, i.e. $End_\mathcal{E} ( \mathds{1} ) = Aut_\mathcal{E} ( \mathds{1} )$.
\item
For any morphism $f : x \to y$ in $\mathcal{E}$, the {\it scalar-multiplication} $\lambda \in End_\mathcal{E} ( \mathds{1}) \mapsto \lambda \cdot f \in  Mor_\mathcal{E} ( x, y)$ (induced by the monoidal structure) is injective.
\end{enumerate}
A typical example of such $\mathcal{E}$ in this paper is given by the symmetric monoidal category $\mathsf{C}_k$ defined in section \ref{202002221748}.

\begin{Defn}
\label{201912260720}
Let $\mathcal{D}$ be a symmetric monoidal category.
Consider a triple $F = (F_o , F_m , \Psi)$ satisfying followings :
\begin{enumerate}
\item
$F_o$ assigns an object $F_o (x)$ of $\mathcal{E}$ to an object $x$ of $\mathcal{D}$.
\item
$F_m$ assigns a morphism $F_m(f) : F_o (x) \to F_o (y)$ of $\mathcal{E}$ to a morphism $f : x \to y$ of $\mathcal{D}$.
\item
$F_m ( Id_x ) = Id_{F_o (x)}$.
\item
For composable morphisms $f,g$ of $\mathcal{D}$, there exists $\lambda \in End_\mathcal{E} ( \mathds{1}_\mathcal{E} )$ such that $\left( F_m(g) \circ F_m(f) \right) = \lambda \cdot F_m (g\circ f)$.
\item
$\Psi$ is a natural isomorphism $\Psi_{x,x^\prime} : F_o (x) \otimes F_o(x^\prime) \to F_o(x \otimes x^\prime)$ which is compatible with the unitors, associators and symmetry of $\mathcal{D}$, $\mathcal{E}$.
\end{enumerate}
We refer to such an assignment $F$ as a {\it symmetric monoidal projective functor from $\mathcal{D}$ to $\mathcal{E}$}.
As a notation, we denote by $F : \mathcal{D} \to \mathcal{E}$.
By an abuse of notations, we also denote by $F_o (x) = F (x)$ and $F_m (f) = F (f)$.
\end{Defn}

\begin{Defn}
\label{201912260925}
Let $F,F^\prime : \mathcal{D} \to \mathcal{E}$ be symmetric monoidal projective functors.
Then a {\it natural isomorphism $\Phi : F \to F^\prime$ in the projective sense} is given as follows :
\begin{enumerate}
\item
For any object $x$ of $\mathcal{D}$, we have an isomorphism $\Phi (x) : F(x) \to F^\prime (x)$ in $\mathcal{E}$.
\item
For a morphism $f : x \to y $ of $\mathcal{D}$, there exists $\lambda, \lambda^\prime \in End_\mathcal{E} ( \mathds{1}_\mathcal{E} )$ such that $\lambda^\prime \cdot F^\prime (f) \circ \Phi (x) = \lambda \cdot \Phi ( y) \circ F (f)$.
\end{enumerate}
If there exists a projective natural isomorphism from $F$ to $F^\prime$, we denote by $F \cong_{\mathrm{proj}} F^\prime$.
A natural isomorphism $\Phi : F \to F^\prime$ of symmetric monoidal projective functors gives {\it a natural isomorphism in the strong sense} if the second condition holds for $\lambda = \lambda^\prime = Id_{\mathds{1}_\mathcal{E}}$.
In that case, we write $F \cong F^\prime$.
\end{Defn}

\begin{Defn}
Let $\mathcal{D}$ be a symmetric monoidal category and $G$ be an abelian group.
For $q \in \mathbb{N}$, {\it a $q$-cochain of $\mathcal{D}$ with coefficients in $G$} is an assignment $\omega ( f_1 , \cdots , f_q ) \in G$ to composable morphisms $f_1 , \cdots , f_q$ in $\mathcal{D}$, i.e. the composition $f_1 \circ \cdots \circ f_q$ is well-defined.
The assignment satisfies
\begin{align}
\omega ( f_1 , \cdots , f_q ) \cdot \omega ( g_1 , \cdots , g_q ) = \omega ( f_1 \otimes g_1 , \cdots , f_q \otimes g_q ) .
\end{align}
We denote by $C^q ( \mathcal{D} ; G)$ the abelian group consisting of $q$-cochains of $\mathcal{D}$ with coefficients in $G$.
A $q$-cochain $\omega$ is {\it normalized} if one of $f_l$ is an identity then $\omega (f_1, \cdots , f_q) =1$.
We define {\it the coboundary map} $\delta : C^q ( \mathcal{D} ; G) \to C^{q+1} ( \mathcal{D} ; G)$ by
\begin{align}\notag
&( \delta \omega ) ( f_1 , \cdots , f_{q+1} ) \\
&= \omega ( f_2 , \cdots , f_{q+1} ) \cdot \omega ( f_1 \circ f_2 , f_3,  \cdots , f_{q+1} )^{-1} \cdots \omega ( f_1, f_2, \cdots f_q \circ f_{q+1} )^{(-1)^q} \cdot \omega (f_1 , \cdots f_q )^{(-1)^{q+1}}  \notag
\end{align}
where $\omega \in C^q ( \mathcal{D} ; G)$.
{\it A $q$-cocycle $\omega$} is a $q$-cochain with $\delta \omega =1$.
It is easy to verify that $\delta \circ \delta$ is trivial.
We define $H^q ( \mathcal{D} ; G ) \stackrel{\mathrm{def.}}{=} \mathrm{Ker} (\delta : C^q \to C^{q+1} ) / \mathrm{Im} ( \delta : C^{q-1} \to C^q )$.
Similarly, we define $H^q_\mathrm{nor} ( \mathcal{D} ; G ) \stackrel{\mathrm{def.}}{=} \mathrm{Ker} (\delta : C^q_\mathrm{nor} \to C^{q+1}_\mathrm{nor} ) / \mathrm{Im} ( \delta : C^{q-1}_\mathrm{nor} \to C^q_\mathrm{nor} )$.
\end{Defn}

\begin{prop}
The induced map $H^2_\mathrm{nor} ( \mathcal{D} ; G ) \to H^2 ( \mathcal{D} ; G )$ is injective.
\end{prop}
\begin{proof}
Suppose that $[ \omega ] \in H^2_\mathrm{nor} ( \mathcal{D} ; G )$ with $[\omega ] \in H^2 ( \mathcal{D} ; G )$ trivial.
Let $\omega$ be a normalized 2-cocycle with $\omega = \delta \theta$ for some 1-cochain $\theta$.
Then it is easy to check $\theta (f) =1$ for an identity $f$.
Hence, $[ \omega] \in H^2_\mathrm{nor} ( \mathcal{D} ; G )$ is trivial.
\end{proof}

\begin{Defn}
\label{202102161148}
Let $F : \mathcal{D} \to \mathcal{E}$ be a symmetric monoidal projective functor.
Note that the automorphism group of the unit object $Aut_\mathcal{E} ( \mathds{1})$ is an abelian group since $\mathcal{E}$ is a symmetric monoidal category.
We define a 2-cochain $\omega (F) \in C^2 ( \mathcal{D} ; Aut_\mathcal{E} ( \mathds{1}) )$ as follows.
For composable morphisms $g,f$ in $\mathcal{D}$, we define $( \omega (F) )(g,f) \in Aut_\mathcal{E} ( \mathds{1})$ by
\begin{align}
\label{201911101800}
F(g) \circ F(f)  = ( \omega (F) ) (g,f) \cdot F(g \circ f ) .
\end{align}
Then the assignment $\omega (F)$ is a well-defined 2-cochain with coefficients in $Aut_\mathcal{E} ( \mathds{1} )$.
\end{Defn}

\begin{prop}
\label{201912252055}
Let $F : \mathcal{D} \to \mathcal{E}$ be a symmetric monoidal projective functor.
Then the 2-cochain $\omega (F)$ is a normalized 2-cocycle.
We define 
\begin{align}
\mathds{O} ( F) \stackrel{\mathrm{def.}}{=} [ \omega (F) ] \in H^2 ( \mathcal{D} ; Aut_\mathcal{E} ( \mathds{1} ) ) .
\end{align}
\end{prop}
\begin{proof}
The 2-cocycle condition is immediate from the associativity of compositions : By the following equality, we obtain $\omega ( f^\prime , f) \cdot \omega ( f^{\prime\prime} , f^\prime \circ f ) \cdot F(f^{\prime\prime} \circ f^\prime \circ f) = \omega ( f^{\prime\prime} ,f^\prime) \cdot \omega ( f^{\prime\prime} \circ f^\prime ,f ) \cdot F(f^{\prime\prime} \circ f^\prime \circ f)$.
\begin{align}
F(  f^{\prime\prime} ) 
\circ
\left( F(  f^\prime ) \circ  F (  f ) \right)
=
\left( F (  f^{\prime\prime} ) \circ  F (  f^{\prime} ) \right) \circ F (  f ) 
\end{align}
By the assumption on $\mathcal{E}$ in the beginning of this section, we obtain 
\begin{align}
\omega ( f^\prime , f) \cdot \omega ( f^{\prime\prime} , f^\prime \circ f)  = \omega (f^{\prime\prime} , f^\prime ) \cdot \omega ( f^{\prime\prime} \circ f^\prime ,  f ) .
\end{align}
It proves that the 2-cochain $\omega$ is a 2-cocycle.
Moreover we have $\omega (F) ( Id_y , f ) = 1 = \omega (F) ( f , Id_x ) $ for any morphism $f: x \to y$ since we have $F(Id_x) = Id_{F(x)}$ for arbitrary object $x$ of $\mathcal{D}$.
Hence, the 2-cocycle $\omega (F)$ is normalized.
It completes the proof.
\end{proof}

\begin{Defn}
\label{202003162027}
Let $F : \mathcal{D} \to \mathcal{E}$ be a symmetric monoidal projective functor.
{\it The coboundary equation associated with $F$} is the following equation of 1-cochain $\theta$ of $\mathcal{D}$ :
\begin{align}
\omega (F) =  \delta \theta .
\end{align}
We define its solution set $\Theta ( F )$ as follows :
\begin{align}
\Theta ( F ) \stackrel{\mathrm{def.}}{=} \{ \theta \in C^1 ( \mathcal{D} ; Aut_\mathcal{E} ( \mathds{1} ) ) ~;~ \omega (F) =  \delta \theta  \} .
\end{align}
An element $\theta \in \Theta (F)$ is called the {\it complementary 1-cochain for a symmetric monoidal projective functor $F$}.
Note that any $\theta \in \Theta (F)$ is normalized.
It is obvious that $\mathds{O} (F) = 1 \in H^2 ( \mathcal{D} ; Aut_\mathcal{E} ( \mathds{1} ) )$ if and only if $\Theta (F) \neq \emptyset$.
\end{Defn}

\begin{Defn}
\label{202001141545}
Let $F : \mathcal{D} \to \mathcal{E}$ be a symmetric monoidal projective functor.
For $\theta \in \Theta (F)$, we define a {\it lift of $F$} by a complementary 1-cochain $\theta$ as a symmetric monoidal functor $( \theta^{-1} \cdot F ) : \mathcal{D} \to \mathcal{E}$ given by 
\begin{align}
(\theta^{-1} \cdot F) (x) &\stackrel{\mathrm{def.}}{=} F(x) \\
(\theta^{-1} \cdot F) (f ) &\stackrel{\mathrm{def.}}{=} \theta (f)^{-1} \cdot F(f) .
\end{align}
The assignment $( \theta^{-1} \cdot F )$ is verified to be a symmetric monoidal functor by definitions.
\end{Defn}

\begin{prop}
\label{201912280733}
Let $F : \mathcal{D} \to \mathcal{E}$ be a symmetric monoidal projective functor.
The induced obstruction is trivial, i.e. $\mathds{O} ( F ) = 1$ if and only if there exists a symmetric monoidal functor $F^\prime : \mathcal{D} \to \mathcal{E}$ such that $F \cong_\mathrm{proj} F^\prime$.
\end{prop}
\begin{proof}
Suppose that $\mathds{O} (F) = 1 \in H^2 ( \mathcal{D} ; Aut_\mathcal{E} ( \mathds{1} ) )$.
By definition of $o(F)$, we choose a normalized 1-cochain $\theta \in \Theta ( F)$.
We define a symmetric monoidal functor $F^\prime = \left( \theta^{-1} \cdot F \right)$
We have a natural isomorphism $F \to F^\prime$ between symmetric monoidal projective functors.
In fact, the identity $Id_{F(x)} : F ( x) \to F(x) = F^\prime (x)$ for any object $x$ of $\mathcal{D}$ gives a natural isomorphism between symmetric monoidal projective functors.
It completes the proof.
\end{proof}

\bibliography{exntension_of_homology_to_tqft}{}

\begin{thebibliography}{10}

\bibitem{baez2009higher}
John~C Baez, Alexander~E Hoffnung, and Christopher~D Walker.
\newblock Higher-dimensional algebra {VII}: {G}roupoidification.
\newblock {\em arXiv preprint arXiv:0908.4305}, 2009.

\bibitem{BK}
Benjamin Balsam and Alexander Kirillov~Jr.
\newblock Kitaev's lattice model and {T}uraev-{V}iro {TQFT}s.
\newblock {\em arXiv preprint arXiv:1206.2308}, 2012.

\bibitem{BW}
John Barrett and Bruce Westbury.
\newblock Invariants of piecewise-linear 3-manifolds.
\newblock {\em Transactions of the American Mathematical Society},
  348(10):3997--4022, 1996.

\bibitem{BKLT}
Yuri Bespalov, Thomas Kerler, Volodymyr Lyubashenko, and Vladimir Turaev.
\newblock Integrals for braided {H}opf algebras.
\newblock {\em Journal of Pure and Applied Algebra}, 148(2):113--164, 2000.

\bibitem{BMCA}
Oliver Buerschaper, Juan~Mart{\'\i}n Mombelli, Matthias Christandl, and Miguel
  Aguado.
\newblock A hierarchy of topological tensor network states.
\newblock {\em Journal of Mathematical Physics}, 54(1):012201, 2013.

\bibitem{DW}
Robbert Dijkgraaf and Edward Witten.
\newblock Topological gauge theories and group cohomology.
\newblock {\em Communications in Mathematical Physics}, 129(2):393--429, 1990.

\bibitem{freed2009remarks}
Daniel Freed.
\newblock Remarks on {C}hern-{S}imons theory.
\newblock {\em Bulletin of the American Mathematical Society}, 46(2):221--254,
  2009.

\bibitem{freed2009topological}
Daniel~S Freed, Michael~J Hopkins, Jacob Lurie, and Constantin Teleman.
\newblock Topological quantum field theories from compact {L}ie groups.
\newblock {\em arXiv preprint arXiv:0905.0731}, 2009.

\bibitem{FQ}
Daniel~S Freed and Frank Quinn.
\newblock Chern-{S}imons theory with finite gauge group.
\newblock {\em Communications in Mathematical Physics}, 156(3):435--472, 1993.

\bibitem{hatcher2002algebraic}
Allen Hatcher.
\newblock {\em Algebraic {T}opology}.
\newblock Cambridge University Press, 2002.

\bibitem{heuts2014ambidexterity}
G~Heuts and J~Lurie.
\newblock Ambidexterity.
\newblock {\em Contemporary Math}, 613:79--110, 2014.

\bibitem{kim2020extension}
Minkyu Kim.
\newblock An extension of {B}rown functor to cospans of spaces.
\newblock {\em arXiv preprint arXiv:2005.10621}, 2020.

\bibitem{kim2020homology}
Minkyu Kim.
\newblock Homology theory valued in the category of bicommutative {H}opf
  algebras.
\newblock {\em arXiv preprint arXiv:2005.04652}, 2020.

\bibitem{kim2019bicommutative}
Minkyu Kim.
\newblock Kitaev's stabilizer code and chain complex theory of bicommutative
  {H}opf algebras.
\newblock {\em arXiv preprint arXiv:1907.09859v4}, 2020.

\bibitem{kim2021integrals}
Minkyu Kim.
\newblock Integrals along bimonoid homomorphisms.
\newblock {\em Applied Categorical Structures}, pages 1--51, 2021.

\bibitem{kirillov2010turaev}
Alexander Kirillov~Jr and Benjamin Balsam.
\newblock Turaev-{V}iro invariants as an extended {TQFT}.
\newblock {\em arXiv preprint arXiv:1004.1533}, 2010.

\bibitem{Kit}
A~Yu Kitaev.
\newblock Fault-tolerant quantum computation by anyons.
\newblock {\em Annals of Physics}, 303(1):2--30, 2003.

\bibitem{larson1988finite}
Richard~G Larson and David~E Radford.
\newblock Finite dimensional cosemisimple {H}opf algebras in characteristic 0
  are semisimple.
\newblock {\em Journal of Algebra}, 117(2):267--289, 1988.

\bibitem{LarSwe}
Richard~Gustavus Larson and Moss~Eisenberg Sweedler.
\newblock An associative orthogonal bilinear form for {H}opf algebras.
\newblock {\em American Journal of Mathematics}, 91(1):75--94, 1969.

\bibitem{monnier2015higher}
Samuel Monnier.
\newblock Higher {A}belian {D}ijkgraaf-{W}itten {T}heory.
\newblock {\em Letters in Mathematical Physics}, 105(9):1321--1331, 2015.

\bibitem{SV}
Amit Sharma and Alexander~A. Voronov.
\newblock Categorification of {D}ijkgraaf-{W}itten theory.
\newblock {\em Advances in Theoretical and Mathematical Physics},
  21(4):1023--1061, 2017.

\bibitem{touze}
Antoine Touz{\'e}.
\newblock On the structure of graded commutative exponential functors.
\newblock {\em International Mathematics Research Notices}, 2018.

\bibitem{TV}
Vladimir~G Turaev and Oleg~Ya Viro.
\newblock State sum invariants of 3-manifolds and quantum 6j-symbols.
\newblock {\em Topology}, 31(4):865--902, 1992.

\bibitem{Wakui}
Michihisa Wakui.
\newblock On {D}ijkgraaf-{W}itten invariant for 3-manifolds.
\newblock {\em Osaka Journal of Mathematics}, 29(4):675--696, 1992.

\end{thebibliography}
\bibliographystyle{plain}

\end{document}